\documentclass[11pt]{amsart}
\usepackage{graphicx}
\usepackage{amsmath}
\usepackage{amssymb, amscd}
\usepackage{epstopdf}
\usepackage[all]{xy}
\newtheorem{theorem}{Theorem}[section]
\newtheorem{lemma}[theorem]{Lemma}
\newtheorem{corollary}[theorem]{Corollary}

\newtheorem{prop}[theorem]{Proposition}
\theoremstyle{definition}

\def\eproof{$\Box$ \medskip}

\def\chb{{\rm Dome}(\Omega)}

\newcommand\rs{\hat{\mathbb{C}}}
\newcommand\Ht{\mathbb{H}^3}
\newcommand\Hp{\mathbb{H}^2}

\title[The Thurston metric and convex hulls]{The Thurston metric on  hyperbolic domains and  boundaries of convex hulls}
\author{Martin Bridgeman}
\address{Boston College}
\author{Richard D. Canary}
\address{University of Michigan}
\date{\today}
\thanks{Bridgeman was partially supported by NSF grant DMS-0707116
and Canary was partially supported by NSF grants DMS-0504791 and DMS-0554239.}

\begin{document}
\begin{abstract} We show that the nearest point retraction is a uniform
quasi-isometry from the Thurston metric on a hyperbolic domain $\Omega\subset\rs$
to the boundary $\chb$ of the convex hull of its complement. As a corollary, one
obtains explicit bounds on the quasi-isometry constant of the nearest point retraction
with respect to the Poincar\'e metric when $\Omega$ is uniformly perfect. We also
establish Marden and Markovic's conjecture that $\Omega$ is uniformly perfect if
and only if the nearest point retraction is Lipschitz with respect to the Poincar\'e metric
on $\Omega$.
\end{abstract}

\maketitle

\section{Introduction}

In this paper we continue the investigation of
the relationship between the geometry of a hyperbolic domain $\Omega$ in $\rs$
and the geometry of the boundary $\chb$ of the convex core of its complement. The surface $\chb$ is hyperbolic in
its intrinsic metric (\cite{EM87,ThBook}) and
the nearest point retraction $r:\Omega\to \chb$ is a conformally natural
homotopy equivalence.
Marden and Markovic \cite{MM} showed
that, if $\Omega$ is uniformly perfect, the nearest point retraction is a quasi-isometry 
with respect to the Poincar\'e metric  on $\Omega$  and
the quasi-isometry constants depend only on the lower bound
for the injectivity radius of $\Omega$ in the Poincar\'e metric.

In this paper, we show that the nearest point retraction is a quasi-isometry with
respect to the Thurston metric  for {\em any} hyperbolic domain $\Omega$ with
quasi-isometry constants which do not depend on the domain.
We recover, as corollaries,
many of the previous results obtained on the relationship between a domain
$\Omega$ and $\chb$ and obtain explicit constants for the first time in several cases. 
We also see that the nearest point retraction is Lipschitz if and only if the
domain is uniformly perfect, which was originally conjectured by Marden
and Markovic \cite{MM}.

We will recall the definition of the Thurston
metric in section \ref{metrics}. It is 2-bilipschitz (see Kulkarni-Pinkall \cite{KP})
to the perhaps more familiar quasihyperbolic metric originally defined by
Gehring and Palka \cite{gehring-palka}. Recall that if $\Omega\subset\mathbb{C}$ and 
$\delta(z)$ denotes
the Euclidean distance from $z\in \Omega$ to $\partial \Omega$, then the quasihyperbolic
metric is simply $q(z)={1\over \delta(z)} dz$.  The Thurston metric has the advantage
of being conformally natural, i.e. invariant under conformal automorphisms of $\Omega$.

\begin{theorem}
\label{projective}
Let $\Omega\subset\rs$ be a hyperbolic domain. Then the nearest point
retraction  \hbox{$r:\Omega \rightarrow \chb$}  is 1-Lipschitz and a $(K,K_0)$-quasi-isometry
with respect to the Thurston 
metric $\tau$ on $\Omega$ and the  intrinsic hyperbolic metric on $\chb$,
where $K\approx 8.49$ and $K_0\approx 7.12$.
In particular,
$$ d_{\chb}(r(z),r(w)) \leq d_{\tau}(z,w) \leq K\ d_{\chb}(r(z),r(w))+K_0.$$
Furthermore, if $\Omega$ is simply connected,  then
$$d_{\chb}(r(z),r(w)) \leq d_{\tau}(z,w) \leq K'\ d_{\chb}(r(z),r(w))+K'_0.$$
where $K'\approx 4.56$ and $K'_0 \approx 8.05$.
\end{theorem}

Marden and Markovic's \cite{MM} proof that the nearest point retraction is
a quasi-isometry in the Poincar\'e metric if $\Omega$ is uniformly perfect
does not yield explicit bounds,
as it makes crucial use of a compactness argument.
Combining Theorem \ref{projective}
with work of Beardon and Pommerenke \cite{BP}, which explains the
relationship between the Poincar\'e metric and the quasihyperbolic
metric, we obtain explicit bounds on the quasi-isometry constants. 
We recall that $\Omega$ is {\em uniformly perfect} if and only if there is
a lower bound for the injectivity radius of $\Omega$ in the Poincar\'e metric.

\begin{corollary}
\label{poincquant}
Let $\Omega$ be a uniformly perfect domain in $\rs$ and
let $\nu>0$ be a lower bound for its injectivity radius in
the Poincar\'e metric,
then $r:\Omega \to\chb$ is  $2\sqrt{2}(k+{\pi^2\over2\nu})$-Lipschitz and a
$(2\sqrt{2}(k+{\pi^2\over2\nu}),K_0)$-quasi-isometry with respect to the Poincar\'e metric.
In particular,
$${1\over 2\sqrt{2}(k+{\pi^2\over 2 \nu})}\ d_{\chb}(r(z),r(w)) \leq d_{\rho}(z,w) \leq K\ d_{\chb}(r(z),r(w))+K_0$$
for all $z,w\in\Omega$ where $k=4+\log (2+\sqrt{2})$.
\end{corollary}

Examples in Section \ref{round}, show that the Lipschitz constant cannot have
asymptotic form better than $O({1\over |\log\nu|\ \nu})$ as $\nu\to 0$. 

When $\Omega$ is simply connected, we obtain:

\begin{corollary}
\label{qiinsc}
Let $\Omega$ be a simply connected domain in $\rs$.
Then
$${1\over 2}\ d_{\chb}(r(z),r(w)) \leq d_{\rho}(z,w) \leq K' \ d_{\chb}(r(z),r(w))+K'_0$$
for all $z,w\in\Omega$.
\end{corollary}

In particular, this gives a simple reproof of Epstein, Marden and Markovic's 
\cite{EMM1} result
that the nearest point retraction is 2-Lipschitz. It only relies on the elementary 
fact that $r$ is 1-Lipschitz (see Lemma \ref{finiteptoh}) with respect to the Thurston
metric and a result of Herron, Ma and Minda \cite{HMM05} concerning the relationship
between the Thurston metric and the Poincar\'e metric in the simply connected setting.

\medskip

Marden and Markovic \cite{MM} also showed that the nearest point retraction
is a quasi-isometry (with respect to the Poincar\'e metric) if and only if $\Omega$
is uniformly perfect. As another corollary of Theorem \ref{projective} we obtain the following 
alternate characterization of uniformly  perfect domains which was  conjectured by
Marden and Markovic. Corollaries \ref{poincquant} and \ref{MMconj} together
give an alternate proof of Marden and Markovic's original characterization.

\begin{corollary}
\label{MMconj}
The  nearest point retraction $r:\Omega\to\chb$ is Lipschitz,
with respect to the Poincar\'e metric on $\Omega$ and the intrinsic
hyperbolic metric on $\chb$ if and only if the set
$\Omega$ is uniformly perfect.   Moreover, if $\Omega$ is not uniformly
perfect, then $r$ is not a quasi-isometry with respect to the Poincar\'e metric.
\end{corollary}

Sullivan \cite{sullivan,EM87} showed that there exists a uniform $M$ such
that if $\Omega$ is simply connected, then there is a conformally natural
$M$-quasiconformal map from $\Omega$ to $\chb$ which admits a bounded
homotopy to the nearest point retraction.
(We say that a map from $\Omega$ to $\chb$
is {\em conformally natural} if it is $\Gamma$-equivariant whenever $\Gamma$ is a group of
M\"obius transformations preserving $\Omega$.) 
Epstein, Marden and Markovic \cite{EMM1,EMM2} showed that the best
bound on $M$  must lie strictly between 2.1 and 13.88.
While, combining work of Bishop \cite{bishop} and Epstein-Markovic
\cite{EM}, the best bound of $M$ lies between
2.1 and 7.82  if
we do not require the quasiconformal map to be conformally natural. 

Marden and Markovic \cite{MM} showed that a domain
is uniformly perfect if and only if there is a quasiconformal map
from $\Omega$ to $\chb$ which is homotopic
to the nearest point retraction by a bounded homotopy.
Moreover, they show that there are bounds
on the quasiconformal constants which depend only on a lower
bound for the injectivity radius of $\Omega$. We note that again these
bounds are not concrete.

In order to use the nearest point retraction to obtain a conformally
natural quasiconformal map between $\Omega $ and $\chb$ 
we show that $r$  lifts to
a quasi-isometry between the universal cover $\tilde\Omega$ of $\Omega$ and
the universal cover ${\rm Dome}(\tilde\Omega)$ of $\chb$. Although
they do not make this explicit in their paper, this is essentially the same approach taken
by Marden and Markovic \cite{MM}.
We recall that quasi-isometric surfaces
need not even be homeomorphic and that even homeomorphic quasi-isometric
surfaces need not be quasiconformal to one another. Notice that the
quasi-isometry constants for the lift of the nearest point retraction are,
necessarily, not as good as those for the nearest point retraction.
Examples in Section \ref{round} show that the first quasi-isometry constant
of $\bar r$
cannot have form better than $O(\nu e^{\pi^2\over 2\nu})$ as $\nu\to 0$.

\begin{theorem}
\label{nprlift}
Suppose that $\Omega$ is a uniformly perfect hyperbolic domain and
$\nu>0$ is a lower bound for its injectivity radius in the Poincar\'e metric. Then
the nearest point retraction lifts to a quasi-isometry
$$\tilde r:\tilde\Omega \to {\rm Dome}(\tilde\Omega)$$
between the universal
cover of $\Omega$ (with the Poincar\'e metric) and the universal cover
of $\chb$ with quasi-isometry constants depending only on $\nu$.  In particular,
$${1\over 2\sqrt{2}(k+{\pi^2\over 2 \nu})}\ d_{{\rm Dome}(\tilde\Omega)}(\tilde r(z),\tilde r(w)) \leq d_{\tilde\Omega}(z,w)\le L(\nu) d_{{\rm Dome}(\tilde\Omega)}(\tilde r(z),\tilde r(w)) + L_0$$
for all $z,w\in \tilde\Omega$ where $L(\nu)=O(e^{\pi^2\over 2\nu})$ as $\nu\to 0$
and $L_0\approx 8.05$.
\end{theorem}

We then apply work of Douady and Earle \cite{douady-earle} to obtain
explicit bound on the quasiconformal map homotopic to the nearest point
retraction when $\Omega$ is uniformly perfect.

\begin{corollary}
\label{upqc}
There exists an explicit function $M(\nu)$ such that
if $\Omega$ is uniformly perfect and $\nu>0$ is a lower bound
for its injectivity radius in the Poincar\'e metric, then there is a conformally natural
$M(\nu)$-quasiconformal map $\phi:\Omega\to\chb$ which admits
a bounded homotopy to $r$.
Moreover, if $\Omega$ is not uniformly perfect, then there does
not exist a bounded homotopy of $r$ to a quasiconformal map.
\end{corollary}

We will give explicit formulas for $L$ and $M$ in section \ref{thelift}.

\bigskip

The table below summarizes the main results concerning the nearest point
retraction in the cases that the domain is simply connected, uniformly perfect
and hyperbolic but not uniformly perfect.

{\scriptsize
\begin{center} \begin{tabular}{ | c | c | c | c | } 
\hline &  $\Omega$ simply connected   &$\Omega$ uniformly perfect & $\Omega$ hyperbolic \\ 
       &  & $\mbox{inj}_{\Omega} > \nu$  & but not uniformly perfect\\ 
\hline Lipschitz bounds on $r$ &  $r$ is 2-Lipschitz &  $r$ is $K(\nu)$-Lipschitz  &   $r$ is not Lipschitz  \\ 
in Poincar\'e metric & (see \cite{EMM1} and Cor. \ref{qiinsc}) & where $K(\nu) = 2\sqrt{2}(k + \frac{\pi^2}{2\nu})$ $k \approx 5.76$ &  (Corollary \ref{MMconj})  \\
& & (Corollary \ref{poincquant}) & \\  
\hline quasi-isometry bounds on $r$ &  $r$ is a $(K',K'_0)$-quasi-isometry & $r$ is a 
$(K(\nu),K_0)$-quasi-isometry &$r$ is not a\\
in Poincar\'e metric & where $K'\approx 4.56$, $K'_0\approx  8.05$ &where $K_0\approx  7.12$  &  quasi-isometry\\  
 & (Corollary \ref{qiinsc}) & (Corollary \ref{poincquant}) & (see \cite{MM} and Cor. \ref{MMconj}) \\
\hline Lipschitz bounds on r &  $r$ is $1$-Lipschitz & $r$ is $1$-Lipschitz & $r$ is $1$-Lipschitz \\ 
in Thurston metric & (Theorem \ref{projective})  &  (Theorem \ref{projective}) &  (Theorem \ref{projective}) \\ 
\hline quasi-isometry bounds on $r$ & $r$ is a $(K',K'_0)$-quasi-isometry & $r$ is a $(K,K_0)$-quasi-isometry & $r$ is a $(K,K_0)$-quasi-isometry \\
in Thurston metric & (Theorem \ref{projective}) & where $K \approx 8.49$ & (Theorem \ref{projective})   \\  
 &  & (Theorem \ref{projective})&\\  

\hline conformally natural  &  $M$-quasiconformal map
 & $M(\nu)$-quasiconformal  map & no quasiconformal map  \\ 
quasiconformal map & where $2.1 <  M < 13.88$  &  (Corollary \ref{upqc}) & (see \cite{MM} and Cor. \ref{upqc})\\
homotopic to $r$ & (see \cite{EMM1} and \cite{EMM2}) & &  \\
by a bounded homotopy &  & & \\
\hline   quasiconformal map &  $M'$-quasiconformal map
 & $M(\nu)$-quasiconformal map & no quasiconformal map  \\ 
homotopic to $r$ & where  $2.1 <  M' < 7.82$ & (Corollary \ref{upqc}) & (see \cite{MM} and Cor. \ref{upqc})  \\
by a bounded homotopy & (see \cite{bishop} and \cite{EM}) & &  \\

\hline
\end{tabular} \end{center}
}

\bigskip

The only constant in the table above which is known to be sharp is the Lipschitz constant in
the simply connected case. It would be desirable to know sharp bounds 
(or better bounds) in the remaining cases. We refer the reader to the excellent
discussion of the relationship between the nearest point retraction and complex
analytic problems, especially Brennan's conjecture, in Bishop \cite{bishop-brennan}.

\bigskip

One of the main motivations for the study of domains
and the convex hulls of their complements comes from the study of hyperbolic
3-manifolds where one is interested in the relationship between the boundary of
the convex core of a hyperbolic manifold and the boundary of its conformal
bordification. 
If $N=\Ht/\Gamma$ is a hyperbolic 3-manifold,
one considers the domain of discontinuity $\Omega(\Gamma)$ for
$\Gamma$'s action on $\rs$. The {\em conformal boundary}
$\partial_cN$ is the quotient $\Omega(\Gamma)/\Gamma$ and 
$\hat N=N\cup\partial_cN$ is the natural conformal bordification of $N$.
The {\em convex core} $C(N)$ of $N$ is obtained as the quotient of  the convex hull
of $\rs-\Omega(\Gamma)$ and its boundary $\partial C(N)$ is the quotient
of $\chb$.
The nearest point retraction descends to a homotopy equivalence
$$\bar r:\partial_cN\to\partial C(N).$$
In this setting, Sullivan's theorem implies that if $\partial_cN$ is  incompressible
in $\hat N$, then there is a quasiconformal map, with uniformly bounded dilatation, 
from $\partial_cN$ to
$\partial C(N)$ which is homotopic to the nearest point retraction. 

Since
the Thurston metric is $\Gamma$-invariant and the nearest point
retraction is $\Gamma$-equivariant we immediately obtain a manifold
version of Theorem \ref{projective}.

\begin{theorem}
\label{projectivemanifold}
Let $N=\Ht/\Gamma$ be a hyperbolic 3-manifold. Then the nearest point
retraction  $r:\partial_cN\to C(N)$  is a $(K,K/K_0)$-quasi-isometry
with respect to the Thurston 
metric $\tau$ on $\partial_cN$ and the  intrinsic hyperbolic metric on $\partial C(N)$.
In particular,
$$ d_{\partial C(N)}(\bar r(z),\bar r(w)) \leq d_{\tau}(z,w) \leq K\ d_{\partial C(N)}(\bar r(z),\bar r(w))+K_0.$$
\end{theorem}

We also obtain a version of Corollary \ref{poincquant}
in the manifold setting. 

\begin{corollary}
\label{poincquantmanifold}
Let $N=\Ht/\Gamma$ be a hyperbolic 3-manifold such that there is a
lower bound $\nu>0$ for the injectivity radius of $\Omega(\Gamma)$
(in the Poincar\'e metric), then
$${1\over 2\sqrt{2}(k+{\pi^2\over 2 \nu})}\ d_{\partial C(N)}(\bar r(z),\bar r(w)) \leq d_{\rho}(z,w) \leq K\ d_{\partial C(N)}(\bar r(z),\bar r(w))+K_0$$
for all $z,w\in\partial_cN$ where $k=4+\log (2+\sqrt{2})$.
\end{corollary}

We also obtain a version of  Corollary \ref{upqc} which guarantees
the existence of a quasiconformal map from the conformal boundary
to the boundary of the convex core. 

\begin{corollary}
\label{qcmanifold}
If $N=\Ht/\Gamma$ is a hyperbolic 3-manifold and
$\nu>0$ is a lower bound for the injectivity radius of $\Omega(\Gamma)$
(in the Poincar\'e metric), then there is a
$M(\nu)$-quasiconformal map $\phi:\partial_cN\to\partial C(N)$ which admits
a bounded homotopy
to the nearest point retraction $\bar r$.
\end{corollary}

If $N$ is a hyperbolic 3-manifold with finitely generated fundamental group,
then its domain of discontinuity is uniformly perfect (see Pommerenke \cite{pommerenke}), so Corollaries \ref{poincquantmanifold} and \ref{qcmanifold}
immediately apply to all hyperbolic 3-manifolds with finitely generated fundamental
group. More generally, analytically finite hyperbolic 3-manifolds
(i.e. those whose conformal boundary has finite area) are known to
be uniformly perfect \cite{canary91}.

We recall that it was earlier shown
(see \cite{BC03}) that the nearest point retraction $\bar r:\partial_cN\to \partial C(N)$ is $(2\sqrt{2}(k+{\pi^2\over 2 \nu}))$-Lipschitz  whenever
$N$ is analytically finite and has a $J(\nu)$-Lipschitz homotopy inverse where
$J(\nu)=O(e^{C\over \nu})$ (for some explicit constant $C>0$) as $\nu\to 0$.
Moreover, if $N=\Ht/\Gamma$ is analytically finite and $\Omega(\Gamma)$ is simply connected,
then $\bar r$ has a $(1+\frac{\pi}{\sinh^{-1}(1)})$-Lipschitz homotopy inverse.
Remark (2) in Section 7 of \cite{BC03} indicates that one cannot improve substantially
on the asymptotic form of the Lipschitz constant for $\bar r$ (see also Section 6 of \cite{cbbc}),
while Proposition 6.2 of \cite{BC03} shows that one cannot improve substantially on the
asymptotic form of the Lipschitz constant of the homotopy inverse for $\bar r$.

\bigskip\noindent{\bf Outline of paper:}
Many of the techniques of proof in this paper are similar
to those in our previous papers \cite{bridgeman} and \cite{BC03}. The consideration
of the Thurston metric allows us to obtain more general results, while
the fact that the Thurston metric behaves well under approximation allows us
to reduce many of our proofs to the consideration of the case where $\Omega$ is
the complement of finitely many points in $\rs$ and the convex hull of its
complement is a finite-sided ideal polyhedron. Therefore, our techniques are more
elementary than much of the earlier work on the subject. 

In Section 2 we recall basic facts and terminology concerning the nearest point
retraction and the boundary of the convex core. In Section 3 we discuss the
Poincar\'e, quasihyperbolic and Thurston metrics on a hyperbolic domain.
In section 4 we establish the key approximation properties which allow us
to reduce the proof of Theorems \ref{projective} and \ref{nprlift} to the case where
$\rs-\Omega$ is a finite collection of points. In Section 5 we study the
geometry of the finitely punctured case and introduce a notion of intersection
number which records the total exterior dihedral angle along a geodesic
in $\chb$. If $\alpha$ is an arc in $\chb$, then the difference between its length
and the length of $r^{-1}(\alpha)$ in the Thurston metric is exactly this
intersection number, see Lemma \ref{finiteptoh}.

Section 6 contains the key intersection number estimates used in the proof of
Theorem \ref{projective}. The first estimate, Lemma \ref{thickintbound}, bounds the
total intersection numbers of moderate length geodesic arcs in the thick part of $\chb$.
(Lemma \ref{thickintbound} is a mild generalization of Lemma 4.3 in \cite{BC03} and
we defer its proof to an appendix.) We then bound the total intersection number of
a short geodesic (Lemma \ref{shortgeointbound}) and the angle of intersection
of an edge of $\chb$ with a short geodesic (Lemma \ref{anglebound}). We give a quick
proof of a weaker form of Lemma \ref{shortgeointbound} and defer the proof of
the better estimate to the appendix. In a remark at the end of Section 6, we see
that one could use this weaker estimate to obtain a version of Theorem \ref{projective}
with slightly worse constants.

Section 7 contains the proof of Theorem \ref{projective}. It follows immediately from Lemma
\ref{finiteptoh} that the nearest point retraction is 1-Lipschitz with respect to the
Thurston metric, but it remains
to bound the distance between two points $x$ and $y$ in $\Omega$ in terms of
the distance between $r(x)$ and $r(y)$ in $\chb$. We do so by first
dividing the shortest geodesic joining $r(x)$ to $r(y)$  into segments of length
roughly $G$ (for some specific constant $G$). The segments in the thick part
have bounded intersection number by Lemma \ref{thickintbound}. We use
Lemmas \ref{shortgeointbound} and \ref{anglebound} to replace geodesic
segments in the thin part by bounded length arcs which consist of one segment
missing every edge and another segment with bounded intersection number.
We then obtain a path joining $r(x)$ to $r(y)$ with length and intersection number
bounded above
by a multiple of \hbox{$d_{\chb}(r(x),r(y))+G$}. Lemma \ref{finiteptoh} can then be applied
to complete the proof.

In Section 8 we derive consequences of Theorem \ref{projective},
including Corollaries \ref{poincquant}, \ref{qiinsc} and  \ref{MMconj}. In Section 9
we prove Theorem \ref{nprlift}, using techniques similar to those in the proof of 
Theorem \ref{projective}. We then use work of Douady and Earle \cite{douady-earle} to derive
Corollary \ref{upqc}. In Section 10, we study the special case of round annuli
and observe that the asymptotic form of our estimates in Corollary
\ref{poincquant} and Theorem \ref{nprlift} cannot be substantially improved.
The appendix, Section 11, contains the complete proofs of Lemmas \ref{thickintbound}
and \ref{shortgeointbound}. The proof of Lemma \ref{thickintbound} is considerably
simpler than the proof of Lemma 4.3 in \cite{BC03} as we are working in the
finitely punctured case.

\section{The nearest point retraction and convex hulls}

If $\Omega \subseteq \rs$ is a hyperbolic domain,  we let  \hbox{$CH(\rs-\Omega)$} be the  hyperbolic convex hull of \hbox{$\rs-\Omega$} in $\Ht$.  
We then define 
$$\chb = \partial CH(\rs-\Omega)$$
and recall that it is hyperbolic in its intrinsic metric (\cite{EM87,ThBook}). In the special case that \hbox{$\rs-\Omega$} lies in a round
circle, \hbox{$CH(\rs-\Omega)$} is a subset of a totally geodesic plane in $\Ht$. In this case,
it is natural to consider $\chb$ to be the double of \hbox{$CH(\rs-\Omega)$} along
its boundary (as a surface in the plane).

We define the {\em nearest point retraction} 
$$r:\Omega\to\chb,$$
by letting $r(z)$ be the point of intersection of the smallest horosphere
about $z$ which intersect $\chb$. We note that \hbox{$r:\Omega\to\chb$} is
a conformally natural homotopy equivalence.
The nearest point retraction extends continuously to a conformally natural map
$$r:\Ht\cup\Omega\to CH(\rs-\Omega)$$ where $r(x)$
is the (unique) nearest point on \hbox{$CH(\rs-\Omega)$} if $x\in \Ht$ (see \cite{EM87}).
  
We recall that a {\em support plane} to $\chb$ is a plane $P$ which intersects $\chb$ and
bounds an open half-space $H_P$ disjoint from $\chb$.
The half-space $H_P$ is associated to a maximal open disc $D_P$
in $\Omega$. Since $D_P$ is maximal,  $\partial D_P$ contains at least two points
in $\partial \Omega$. Therefore, $P$ contains a geodesic $g \in \chb$. 
Notice that  $P$ intersects $\chb$ in either a single edge or a convex subset of $P$
which we call a {\em face} of $\chb$. If $z\in \Omega$, then the geodesic ray
starting at $r(z)$ in the direction of $z$ is perpendicular to a support plane $P_z$
through $r(z)$.

\section{The Poincar\'e metric, the quasihyperbolic metric and the Thurston metric}
\label{metrics}

Given a hyperbolic domain $\Omega$ in $\rs$  there is a unique metric 
$\rho$ of constant curvature $-1$ which is conformally equivalent to the 
Euclidean metric on $\Omega$. 
This metric $\rho$ is called the {\em Poincar\'e metric} on $\Omega$ and  is
denoted 
$$\rho(z) = \lambda_\rho(z)\ dz.$$
Alternatively, the Poincar\'e metric can be given by defining the length of a tangent
vector \hbox{$v \in T_x(\Omega)$} to be the infimum of the hyperbolic length over  all vectors
$v'$ such that there exists a conformal map \hbox{$f:\Hp \rightarrow \rs$} such that
\hbox{$f(\Hp)\subset \Omega$} and  \hbox{$df(v') = v$.}

In \cite{BP}, Beardon and Pommerenke show that the quasihyperbolic metric and the Poincar\'e metric are closely related. We recall that,
the {\em quasihyperbolic  metric} on $\Omega\subset \mathbb{C}$ is the conformal metric 
$$q(z) = \frac{1}{\delta(z)}\ dz$$
where $\delta(z)$ is the Euclidean distance from $z$ to $\partial \Omega$. 
They let 
$$\beta(z)=\inf\left\{\left|\log {|z-a|\over |z-b|}\right|\ : a\in\partial\Omega,\ b\in\partial\Omega,\ |z-a|=\delta(z)\right\}$$
Alternatively, one can define \hbox{$\beta(z)=\pi M$}  where $M$ is the maximal
modulus of an essential round annulus  (i.e. one which separates \hbox{$\rs-\Omega$}) in $\Omega$
whose central (from the conformal viewpoint) circle passes through $z$. 
We recall that if $A$ is any annulus in $\rs$, then it is conformal to an annulus
of the form \hbox{$\{ z\ |\ {1\over t}< |z| <t\}$} and we define its {\em modulus} to be
$${\rm mod}(A)={1\over 2\pi}\log(t^2).$$
The central circle of $A$ is the pre-image of
the unit circle in $\{ z\ |\ {1\over t}< |z| <t\}$.

\begin{theorem}{\rm (Beardon-Pommerenke \cite{BP})}
\label{BPresult}
If $\Omega$ is a hyperbolic domain in $\mathbb{C}$, then
$${1\over \sqrt{2}(k+\beta(z))}q(z)\le \rho(z) \le {2k+{\pi\over 2}\over k+\beta(z)}q(z)$$
for all $z\in\Omega$, where
$k=4+\log(3+2\sqrt{2})\approx 5.76$.
Moreover,
$$\rho(z)\le 2q(z)$$
for all $z\in \Omega$.
If $\Omega$ is simply connected, then
$${1\over 2} q(z) \le \rho(z) \le 2q(z).$$
\end{theorem}

Using the fact that the core curves of large modulus essential annuli are short
in the Poincar\'e metric, one sees that a lower bound on the injectivity
radius of $\Omega$ gives an explicit bilipschitz equivalence between
the Poincar\'e metric and the quasihyperbolic metric when $\Omega$ is uniformly perfect.

\begin{corollary}
\label{poincandqh}
If $\Omega$ is a uniformly perfect hyperbolic domain and $\nu>0$ is
a lower bound for the injectivity radius of $\Omega$ in the Poincar\'e
metric, then 
$${1\over \sqrt{2}(k+{\pi^2\over 2\nu})}q(z)\le  \rho(z)\le 2q(z)$$
for all $z\in\Omega$.
\end{corollary}

\begin{proof}
Corollary 3.3 in \cite{cbbc} asserts that if 
\hbox{$\beta(z)\ge M$}, then \hbox{${\rm inj}_\Omega(z)\le {\pi^2\over 2M}$.}
Therefore, if \hbox{${\rm inj}_\Omega(z)\ge \nu>0$},
then \hbox{$\beta(z)\le {\pi^2\over 2\nu}$}. The result then follows from
Theorem \ref{BPresult}.
\end{proof}

\bigskip

Another metric closely related to the quasihyperbolic metric is the Thurston metric.
The {\em Thurston metric} 
$$\tau(z)=\lambda_\tau(z)\ dz$$
on $\Omega$ is defined by letting the length of a vector
\hbox{$v \in T_{x}(\Omega)$} be the infimum of the hyperbolic length of all vectors
\hbox{$v' \in \Hp$} such that there exists a M\"obius transformation 
$f$ such that \hbox{$f(\Hp)\subset \Omega$} and \hbox{$df(v') = v$}. By definition, the Thurston
metric is conformally natural and conformal to the Euclidean metric.
Furthermore, by our alternate 
definition of the Poincar\'e metric, we have the inequality 
\begin{equation}
\rho(z) \leq \tau(z)
\label{rlesst}
\end{equation}
for all $z\in\Omega$.

The Thurston metric is also known as the projective
metric, as it arises from regarding $\Omega$ as a complex projective surface
and giving it the metric Thurston described on such surfaces. 
Kulkarni and Pinkall defined and studied an
analogue of this metric in all dimensions and it is also sometimes called the 
Kulkarni-Pinkall metric. See Herron-Ma-Minda \cite{HMM05}, Kulkarni-Pinkall \cite{KP},
McMullen \cite{mcmullenCE} and Tanigawa \cite{tanigawa} for further information
on the Thurston metric. 

Kulkarni and Pinkall (see Theorem 7.2 in \cite{KP}) proved that the Thurston metric and
the quasihyperbolic metric are 2-bilipschitz.

\begin{theorem}{\rm (\cite{KP})}
\label{projtoqh}
If $\Omega$ is a hyperbolic domain in $\mathbb{C}$ then
$$ {1\over 2} \tau(z) \leq q(z) \leq \tau(z).$$
\end{theorem}

In the simply connected setting, Herron, Ma and Minda
(see Theorem 2.1 and Lemma 3.1 in \cite{HMM05}) prove
that the Poincar\'e metric and the Thurston metric are also 2-bilipschitz.

\begin{theorem}{\rm (\cite{HMM05})}
\label{scThurPoin}
If  $\Omega$ is a simply connected hyperbolic domain in $\rs$,
then 
$$ {1 \over 2}\tau(z) \leq \rho(z) \leq \tau(z).$$
\end{theorem}

We will make use of the following alternate characterization of the Thurston metric.
We note that it is not difficult to derive Theorem \ref{projtoqh} from this
characterization.

\begin{lemma}
\label{projhoro}
If $\Omega$ is a hyperbolic domain in $\mathbb{C}$, then
the Thurston metric $\tau(z)$ on $\Omega$ is given by
$$\tau(z)= \frac{1}{h(z)}\ dz$$
where $h(z)$ is the radius of the maximal horoball $B_z$ based at $z$
whose interior misses $\chb$. 
\end{lemma}

\begin{proof}
Recall the disk model $\Delta$ for $\Hp$ with the metric 
$$\rho_{\Delta}(z) = \frac{2}{1-|z|^{2}}\ dz  = \lambda_{\Delta}(z)\ dz .$$
By normalizing, we may define the length of a vector \hbox{$v \in T_{x}(\Omega)$} in the Thurston metric to be the infimum of the length of vectors \hbox{$v' \in T_{0}\Delta$} such that there exists a M\"obius transformation $m$ such that \hbox{$m(\Delta)\subset \Omega$} with $dm_0(v') = v$.
The hyperbolic length of $v'$ is  then simply \hbox{$2|v'| = \frac{2}{|m'(0)|}|v|$.}
Thus
$$\tau(z) = \inf_{m} \frac{2}{|m'(0)|}$$
where the infimum is taken over all M\"obius transformations $m$ such that 
$m(\Delta) \subset \Omega$ and $m(0) = z$.  

Given such a M\"obius transformation $m$,
we let  \hbox{$D_{m} = m(\Delta)$} and let $P_{m}$ be the hyperbolic plane with the same
boundary as $D_{m}$.
We further let $B_{m}$ be the horoball based at $z$ which is tangent to $P_{m}$. As 
\hbox{$D_{m} \subseteq \Omega$}, $P_{m}$ does not intersect the interior of  the convex hull \hbox{$CH(\rs-\Omega)$}. Similarly, the interiors of $B_{m}$ and \hbox{$CH(\rs-\Omega)$} are disjoint. 
Let $h_{m}$ be the radius of the horoball $B_{m}$. 
Notice that $h_m\le h(z)$.
Let  $n$ be a parabolic M\"obius transformation with fixed point $z$,
which sends $B_{m}\cap P_m$ to the point vertically above $z$.
Then, $n$ fixes $B_{m}$ and sends $D_{m}$ to a disk  centered at $z$
with radius $2h_m$.
Therefore, since $n\circ m$ sends a disk of radius $1$ to a disk of radius $2h_{m}$
and sends centers to centers,  
$$|n\circ m'(0)| = |n'(z)m'(0)| = 2h_{m}.$$
Since $n$ is parabolic and fixes $z$,  \hbox{$|n'(z)| = 1$}, which implies
that \hbox{$|m'(0)|=2h_m$.} 
Therefore
$$\tau(z) = \inf_{m} \frac{2}{|m'(0)|} = \inf_{m} \frac{1}{h_{m}}\ge {1\over h(z)}$$

On the other hand, let $P_z$ be the support plane to $r(z)$ which is perpendicular
to the geodesic ray joining $r(z)$ to $z$. Then $P_z$
is tangent to $B_z$ at $r(z)$. Let $D_z$ be the (open) disk in $\Omega$ 
containing $z$ whose
boundary agrees with $\partial P_z$. If $m$ is a M\"obius transformation
such that $m(0)=z$ and $m(\Delta)=D_z$, then, by the above
analysis, $|m'(0)|=2h(z)$. It follows that
$$\tau(z) = {1\over h(z)}$$
as desired.
\end{proof}

\section{Finite approximations and the geometry of domains and
convex hulls}

Given a hyperbolic domain $\Omega$, let \hbox{$\{x_n\}_{n=1}^\infty$} be a dense set of points
in \hbox{$\rs-\Omega$}. If \hbox{$X_n=\{ x_1,\ldots, x_n\}$}, then we call
$\{ X_n\}$ a {\em nested finite approximation} to \hbox{$\rs- \Omega$}.
Notice that $\{X_n\}$ converges to $\rs-\Omega$ in the
Hausdorff topology on closed subsets of $\rs$. 
We let \hbox{$\Omega_n = \rs - X_n$},
denote the quasihyperbolic metric on $\Omega_n$ by  $q_n$, and let 
\hbox{$r_n: \Omega_n \rightarrow {\rm Dome}(\Omega_n)$} be the nearest point retraction.
Since \hbox{$\Omega \subseteq \Omega_n$}, $q_n$ and $r_n$ are both defined on
$\Omega$.  Results of Bowditch \cite{bowditch} imply that $\{CH(X_n)\}$
converges to \hbox{$CH(\rs-\Omega)$} and that \hbox{$\{ {\rm Dome}(\Omega_n)\}$} converges
to $\chb$ in the Hausdorff topology on closed subsets of $\Ht$.

The following approximation lemma will allow us to reduce the
proof of many of our results to the case where $\Omega$ is the complement
of finitely many points.

\begin{lemma}
\label{quasiapprox}
Let $\{ X_n\}$ be a nested finite approximation to $\rs-\Omega$. 
If, for all $n$,
$\Omega_n=\rs-X_n$, $q_n$ is the quasihyperbolic metric on $\Omega_n$,
and $p_n$ is the Thurston metric on $\Omega_n$, then
\begin{enumerate}
\item
the sequence $\{ q_n\}$ of metrics  converges uniformly on compact subsets of
$\Omega$ to  the quasihyperbolic metric $q$ on $\Omega$,
\item
the sequence $\{ \tau_n\}$ of metrics  converges uniformly on compact subsets of
$\Omega$ to  the Thurston metric $\tau$ on $\Omega$, 
\item
the associated
nearest point retractions
\hbox{$\{r_n: \Omega_n \rightarrow {\rm Dome}(\Omega_n)\}$} converge to the nearest point retraction
$r: \Omega \rightarrow \chb$ uniformly on compact subsets of $\Omega$,
\item
if $z,w\in\Omega$, then
\hbox{\{$d_{{\rm Dome}(\Omega_n)}(r_n(z),r_n(w))\}$} converges to 
\hbox{$d_{\chb}(r(z),r(w))$}, and
\item
if $z\in\Omega$, then 
$$\lim {\rm inj}_{{\rm Dome}(\Omega_n)}(r_n(z))={\rm inj}_{\chb}(r(z)).$$
\end{enumerate}
\end{lemma}

\begin{proof}
Let $\delta_n(z)$ denote the distance from \hbox{$z\in\Omega_n$} to $X_n$. It is clear that
$\{\delta_n\}$ converges to $\delta$ uniformly on compact subsets of
$\Omega$. Therefore, $\{ q_n\}$ converges uniformly to $q$ on compact subsets
of $\Omega$.

If $z\in \Omega_n$, let $B^n_z$ denote the maximal horoball based at $z$
whose interior is disjoint from $CH(X_n)$ and whose closure
intersects ${\rm Dome}(\Omega_n)$ at $r_n(z)$.
Suppose that \hbox{$\{z_n\}\subset\Omega$}
converges to $z\in\Omega$.  Since $\{CH(X_n)\}$ converges to
\hbox{$CH(\rs-\Omega)$}, $\{B^n_{z_n}\}$ converges to a horoball $B$ based at $z$ whose
interior is disjoint from \hbox{$CH(\rs-\Omega)$} and whose closure intersects $\chb $
at \hbox{$\lim r_n(z_n)$}. By definition, $r(z)$ is this intersection point, so
\hbox{$r(z)=\lim r_n(z_n)$}. Therefore, $\{r_n\}$ converges uniformly to $r$ on $\Omega$,
establishing (3). Moreover,  applying
Lemma \ref{projhoro}, we see that $\{ \tau_n(z_n)\}$ converges to $\tau(z)$,
so $\{ \tau_n\}$ converges uniformly to $\tau$ on $\Omega$, which proves (2).

Suppose that $z,w\in\Omega$.  If $\gamma_n$ is a path in ${\rm Dome}(\Omega_n)$
of length 
$$l_n=d_{{\rm Dome}(\Omega_n))}(r_n(z),r_n(w)),$$
then $\gamma_n$ converges, up to subsequence,
to a path in $\chb$ of length at most $\liminf l_n$ joining $r(z)$ to $r(w)$.
Therefore,
$$d_{\chb}(r(z),r(w))\le\liminf l_n.$$
On the other hand,
if $\gamma$ is a path in $\chb$
of length \hbox{$l=d_\chb(r(z),r(w))$}, then $r_n(\gamma)$ is a path of length at most $l$
on ${\rm Dome}(\Omega_n)$ joining $r_n(r(z))$ to $r_n(r(w))$. Notice that
\hbox{$\{d_{ {\rm Dome}(\Omega_n)}(r_n(r(z)),r_n(z))\}$} and 
\hbox{$\{d_{ {\rm Dome}(\Omega_n)}(r_n(r(w)),r_n(w))\}$} both converge to 0,
since \hbox{$\{{\rm Dome}(\Omega_n)\}$} converges to $\chb$. Therefore,
\hbox{$\limsup l_n\le l$.} It follows that $\lim l_n=l$, establishing (4).

The proof of (5) follows much the same logic as the proof of (4).


\end{proof}

\section{The geometry of finitely punctured spheres}
\label{fingeom}

In this section we consider the situation where $\Omega$ is the complement of
finitely many points in $\rs$. In this case, the convex core \hbox{$CH(\rs-\Omega)$}
of the complement of $\Omega$ is an ideal
hyperbolic polyhedron and its boundary $\chb$ consists of a finite
number of ideal hyperbolic polygons  meeting along edges. We give a precise
description of the Thurston metric in this case (Lemma \ref{finiteproj}) and show
that $r$ is 1-Lipschitz with respect to the Thurston metric (Lemma \ref{finiteptoh}).
We also introduce an intersection number for a curve 
$\alpha:[0,1]\to\chb$ which is transverse to all edges and
give a formula for the length of its pre-image $r^{-1}(\alpha)$ in the Thurston
metric on $\Omega$.

\subsection{The Thurston metric in the finitely punctured case}

Let $\Omega$
be the complement of a finite set in $\rs$. Each face $F$ of the polyhedron
\hbox{$CH(\rs-\Omega)$}  is contained in a unique plane $P_F$.
Furthermore, $P_F$ bounds a unique open half-space $H_F$ which does not intersect 
\hbox{$CH(\rs-\Omega)$}. The half-space $H_F$ defines an open disk $D_F$ in $\Omega$.  
Let $F'$ be the ideal polygon in $D_F$ with same  vertices as $F$. Abusing
notation, we will call $F'$ a {\em face} of $\Omega$.
If two faces $F_1$ and $F_2$ meet in an edge $e$ with exterior dihedral angle
$\theta_e$ then 
the spherical bigon \hbox{$B_e=D_{F_1}\cap D_{F_2}$}
is adjacent to the polygons $F'_1$ and $F'_2$ and has angle $\theta_e$. 
Thus we have a decomposition of $\Omega$ into faces $F'$  and spherical bigons $B$ corresponding to the edges of \hbox{$CH(\rs-\Omega)$}. 

If $F'$ is a face of $\Omega$ then it inherits a hyperbolic metric $p_{F'}$
which is the restriction of the Poincar\'e metric on $D_F$. Each spherical
bigon $B_e$ inherits a Euclidean metric in the following manner.
Let $f_e$ be a M\"obius transformation taking $B_e$ to
the Euclidean rectangle \hbox{$S_e={\bf R}\times[0, \theta_e]$}.
We then give $B_e$ the Euclidean metric
$p_e$ obtained by pulling back the standard Euclidean
metric on $S_e$. Since any M\"obius transformation preserving $S_e$ is a Euclidean isometry,
this metric is independent of our choice of $f_e$.

Lemma \ref{finiteproj} shows that the Thurston metric
agrees with the metrics defined above.
One may view this as a special case of the discussion in section 2.1
of Tanigawa \cite{tanigawa}, see also section 8 of Kulkarni-Pinkall \cite{KP} and
section 3 of McMullen \cite{mcmullenCE}.

\begin{lemma}
\label{finiteproj}
Let $\Omega$ be the complement of finitely many points in $\rs$. Then
\begin{enumerate}
\item The Thurston metric on $\Omega$ agrees with $\rho_{F'}$ on
each face $F'$ of $\Omega$ and agrees with $p_e$ on each spherical
bigon $B_e$.
\item With respect to the Thurston metric,
the nearest point retraction $r:\Omega \rightarrow \chb$ maps each face $F'$ isometrically to the face $F$ of $\chb$ and is Euclidean projection on each bigon.
That is,  if $B_e$ is a bigon, there is an isometry $g_e:e\rightarrow {\bf R}$
such that 
$$\begin{CD}
B @>r>> e\\
@VV f_e V  @VV g_e V\\
{\bf R}\times[0,\theta] @>\pi_1>> {\bf R}
\end{CD}
$$
\end{enumerate}
\end{lemma}

{\bf Proof:}
First suppose that $F'$ is a face of $\Omega$. We may normalize by
a M\"obius transformation so that $D_F$ is the unit disk and $z=0$.
In this case, the maximal horoball at $z$ has height 1, so
the Thurston metric for $\Omega$ at $z$ agrees with the Poincar\'e metric
on $D_F$ at $z$. Moreover, one sees that $r$ agrees with the hyperbolic
isometry from the unit disk to $P_F$ on $F'$.

Now let $e$ be an edge between two faces $F_1$ and $F_2$.
We may normalize so that $e$ is  the vertical line joining $0$ to $ \infty$ and define the isometry
$g_e$ by \hbox{$g_e(0,0,t) = \log{t}$}. We  may further assume that  $B_e$
is the sector \hbox{$\{se^{i t}\ | \ s > 0,\ 0 \leq t \leq \theta\}.$} For a point \hbox{$z = se^{it} \in B$},
the maximal horoball at $z$ is tangent to $e$ at height $s$.
Therefore the Thurston metric on $B$ is $\frac{1}{s}dz$,
which agrees with $p_e$. Moreover, \hbox{$r(se^{it})=s$},
so $r$ has the form described above.
\mbox{ \eproof}

\subsection{Intersection number and the nearest point retraction}\hfill
If  \hbox{$\alpha:[0,1] \rightarrow  \chb$} is a curve  on 
$\chb$ which is transverse to the edges of $\chb$,
then we define $i(\alpha)$ to be the sum of the dihedral angles at each of its intersection points with an edge. We allow the endpoints of $\alpha$ to lie on edges.
(In a more general setting, $i(\alpha)$ is thought of as the
intersection number of $\alpha$ and the bending lamination 
on $\chb$, see \cite{EM87}.) We denote by $l_h(\alpha)$ the length of $\alpha$ in
the intrinsic metric on $\chb$ (which is also its length in
the usual hyperbolic metric on $\Ht$).

If $\gamma:[0,1]\to\Omega$ is a rectifiable curve in $\Omega$, 
then let  $l_\tau(\gamma)$ denote its length in
the Thurston metric. Similarly, we can define $l_q(\gamma)$ to be its
length in the quasihyperbolic metric and $l_\rho(\gamma)$ to be its length in
the Poincar\'e metric. (Since all these metrics are complete and locally bilipschitz,
a curve is rectifiable in one metric if and only if it is rectifiable in another.)

\begin{lemma}
Suppose that $\Omega$ is the complement of finitely many points in $\rs$.
\begin{enumerate}
\item If $\gamma:[0,1] \rightarrow \Omega$ then  
$$ l_\tau(\gamma) \geq l_h(r\circ\gamma).$$
\item
If $\alpha:[0,1] \rightarrow \chb$ is a curve transverse to the edges of $\chb$ then
$$l_\tau(r^{-1}(\alpha)) = l_h(\alpha) +   i(\alpha).$$
\end{enumerate}
\label{finiteptoh}
\end{lemma}

\begin{proof}
By Lemma \ref{finiteproj}, $r$  is an isometry on the faces of $\chb$
and  Euclidean projection on the bigons. It follows that $r$ is 1-Lipschitz. Thus, for
$\gamma:[0,1] \rightarrow \Omega$,
$$l_\tau(\gamma) \geq l_h(r\circ\gamma).$$

Now let $\alpha:[0,1] \rightarrow \chb$ be a curve transverse to the edges of $\chb$.
Let $\{x_i = \alpha(t_i)\}_{i=1}^{n-1}$ be the finite collection of intersection points of
$\alpha$ with the edges of $\chb$. Let $t_0=0$ and $t_n=1$. (Notice that if $\alpha(0)$ lies
on an edge, then \hbox{$t_0=t_1$}, while \hbox{$t_n=t_{n-1}$} if $\alpha(1)$ lies on an edge.)

Let $\alpha_i$ be the restriction of $\alpha$ to \hbox{$(t_{i-1},t_{i})$}.
Then
$$l_\tau(r^{-1}(\alpha)) =
\sum_{i=1}^n l_\tau(r^{-1}(\alpha_i)) + \sum_{i=1}^{n-1} l_\tau(r^{-1}(x_i))$$
As $r$ is an isometry on the interior of the faces, 
$$\sum_i l_\tau(r^{-1}(\alpha_i)) = \sum_i l_h(\alpha_i) = l_h(\alpha).$$

If $x_i$ lies on the edge $e_i$ of $\chb$ with exterior dihedral angle $\theta_i$,
then, by Lemma \ref{finiteproj},  $r^{-1}(x_i)$ has length $\theta_i$
in the Thurston metric. Therefore, 
$$ \sum_{i=1}^{n-1} l_\tau(r^{-1}(x_i)) = \sum_{i=1}^{n-1} \theta_i = i(\alpha).$$
Thus,
 $$l_\tau(r^{-1}(\alpha)) = l_h(\alpha) +   i(\alpha)$$
as claimed.
\end{proof}

\noindent
{\bf Remark:} Part (2)  is closely related to
Theorem 3.1 in McMullen \cite{mcmullenCE} which shows
that 
$$l_\rho(r^{-1}(\alpha))\le l_h(\alpha)+i(\alpha).$$

\section{Intersection number estimates}

We  continue to assume throughout this section
that $\Omega$ is the complement of finitely many points in $\rs$.
We obtain bounds on $i(\alpha)$ for short geodesic arcs
in the thick part (Lemma \ref{thickintbound}) and short
simple closed geodesics in $\chb$ (Lemma \ref{shortgeointbound}). We use 
Lemmas \ref{thickintbound} and \ref{shortgeointbound} to bound the angle of intersection between an edge of $\chb$
and a short simple closed geodesic (Lemma \ref{anglebound}).

\subsection{Short geodesic arcs in the thick part}

We first state a mild generalization of Lemma 4.3 from \cite{BC03}
which gives the bound on $i(\alpha)$ for short geodesic arcs. 
In \cite{BC03} we define a function
$$F(x) = \frac{x}{2} + \sinh^{-1}\left(\frac{\sinh\left( {x\over 2}\right)}{\sqrt{1-\sinh^{2}\left({x\over 2}\right)}}\right)$$
and we let \hbox{$G(x) = F^{-1}(x)$.}  The function $F$ is
monotonically increasing and  has domain \hbox{$(0, 2 \sinh^{-1}(1))$.} 
The function $G(x)$ has domain $(0,\infty)$, has asymptotic behavior
\hbox{$G(x) \asymp x$} as $x $ tends to $0$, and $G(x)$ approaches $ 2\sinh^{-1}(1)$
as $x$ tends to $ \infty$. 

\begin{lemma}
\label{thickintbound}
Suppose that $\Omega$ is the complement of finitely many points in $\rs$. 
If \hbox{$\alpha:[0,1] \rightarrow \chb$} is a  geodesic path 
and
$$l_{h}(\alpha) \le G({\rm inj}_{\chb}(\alpha(t))),$$
for some $t \in[0,1]$, then
$$i(\alpha) \le 2\pi.$$
\end{lemma} 

In the appendix, we give a self-contained proof of Lemma \ref{thickintbound}.
The proof uses the same basic techniques as in \cite{BC03},
but the arguments are much more elementary, as we
need not use the general theory of bending measures. 

\subsection{Cusps and collars}

In this section, we recall a version of the Collar Lemma (see Buser \cite{buser})
which gives a complete description of the portion of a hyperbolic
surface with injectivity radius less than $\sinh^{-1}(1)$.

If $\gamma$ is a simple closed geodesic in a hyperbolic surface $S$,
then we define the {\em collar} of $\gamma$ to be the set 
$$N(\gamma) = \{x\ | d_S(x,\gamma) \leq w(\gamma)\} \mbox{ where }
w(\gamma) =  \sinh^{-1}\left(\frac{1}{\sinh\left(\frac{l_{S}(\gamma)}{2}\right)}\right).$$

A {\em cusp} is a subsurface isometric to $S^1\times [0,\infty)$ with the metric
$$\left({e^{-2t}\over \pi^2}\right)d\theta^2+dt^2.$$

\medskip\noindent
{\bf Collar Lemma:} {\em 
Let $S$ is a complete hyperbolic surface.
\begin{enumerate}
\item The collar $N(\gamma)$  about a simple closed geodesic $\gamma$ is
isometric to ${\bf S}^1 \times[-w(\gamma_i),w(\gamma_i)]$ with
the metric $$\left({l_S(\gamma_i)\over 2\pi}\right)^2\cosh^2{t}\ d\theta^2+dt^2.$$
\item A collar about a simple closed geodesic is disjoint from the collar
about any disjoint simple closed geodesic and from any cusp of $S$
\item Any two cusps in $S$ are disjoint.
\item If ${\rm inj}_S(x)\le \sinh^{-1}(1)$, then $x$ lies in a cusp or in
a collar about a simple closed geodesic of length at most $2\sinh^{-1}(1)$.
\item If $x \in N(\gamma)$, then
$$\sinh({\rm inj}_S(x)) = \sinh(l_S(\gamma)/2)\cosh(d_S(x,\gamma_i)).$$
\end{enumerate}
}

If $C$ is a cusp or a collar about a simple closed geodesic, then we will
call a curve of the form $S^1\times \{t\}$, in the coordinates given above,
a {\em cross-section}.  The following observations concerning the intersections
of cross-sections with edges will be useful in the remainder of the section. 

\begin{lemma}
\label{cuspcross}
Suppose that $\Omega$ is the complement of a finite collection $X$ of
points in $\rs$ and $C$ is a cusp of $\chb$. Then every edge which
intersects $C$ intersects every cross-section of $C$ and does so orthogonally.
Moreover,
if $\gamma$ is any cross-section of $C$, then $i(C)=2\pi$.
\end{lemma}

\begin{proof}
Identify the universal cover ${\rm Dome}(\tilde\Omega)$ of $\chb$ with $\Hp$
so that the pre-image of $C$ is the horodisk \hbox{$H=\{z\ |\ Im(z)\ge 1\}$,}
which is invariant under the covering transformation \hbox{$z\to z+2$.}
If $\tilde e$ is the pre-image of an edge, then $\tilde e$ either terminates
at $\infty$,  or has two finite endpoints. If $\tilde e$  terminates at infinity then  it crosses every cross-section of $H$ orthogonally. Otherwise, since the edge $e$ is simple,
the endpoints of $\tilde e$  must have Euclidean distance less than 2 from one another and therefore $\tilde e$ does not intersect $H$ (see also Lemma 1.24 in \cite{buff}).
 
Every cusp is associated to a point $x\in X$, which we may regard as an ideal vertex
of $CH(X)$.  If we move $x$ to $\infty$,
then each cross-section is a Euclidean polygon in a horosphere, so has
total external angle exactly $2\pi$.
\end{proof}

\begin{lemma}
\label{collarcross} 
Suppose that $\Omega$ is the complement of finitely many points in $\rs$ and $N(\gamma)$ is collar about a simple closed geodesic $\gamma$ on $\chb$. If $e$ is an edge of $\chb$, then 
each component of $e\cap N(\gamma)$ intersects every cross-section of $N(\gamma)$ exactly once. 
\end{lemma}

\begin{proof}
Again identify the universal cover of $\chb$ with $\Hp$.
Let $\tilde \gamma$ be a component of the pre-image of $\gamma$ and
let $N(\tilde\gamma)$ be the metric neighborhood of $\tilde\gamma$ of
radius $w(\gamma)$. Then $N(\tilde\gamma)$ is a component of
the pre-image of $N(\gamma)$. Notice that $\tilde\gamma$ is the axis of a hyperbolic
isometry $g$ in the conjugacy class determined by the simple closed curve $\gamma$.

Notice that every edge $e$ of $\chb$ is an infinite
simple geodesic which does not accumulate on $\gamma$.
Let $\tilde e$ be the pre-image of an edge which intersects $N(\tilde\gamma)$, then
neither of its endpoints can be an endpoint of $\gamma$.
If the endpoints of $\tilde e$ lie on the same side of $\tilde \gamma$, one
may use the fact that they must lie in a single fundamental domain for $g$ and
the explicit description of $N(\tilde\gamma)$
to show that $\tilde e$ cannot intersect $N(\tilde\gamma)$ (see also Lemma 1.21 in
\cite{buff}).
If its endpoints lie on opposite sides of $\tilde\gamma$, then it intersects
every cross-section of $N(\tilde\gamma)$ exactly once, as desired.
\end{proof}

We now observe that if $\gamma$ is  a homotopically non-trivial curve on $\chb$,
then \hbox{$i(\gamma)\ge 2\pi$.}

\begin{lemma}
\label{intlowerbound}
If $\Omega$ is the complement of finitely many points in $\rs$ and 
\hbox{$\gamma:S^1\to \chb$} is homotopically non-trivial in $\chb$, then
$$i(\gamma)\ge 2\pi.$$
\end{lemma}

\begin{proof}
If $\gamma$ is a closed geodesic, then, by the Gauss-Bonnet theorem, its geodesic
curvature is at least $2\pi$. But $i(\gamma)$
is the sum of the dihedral angles of the edges $\gamma$ intersects, and is therefore 
at least as large as the geodesic curvature of $\gamma$, so $i(\gamma)\ge2\pi$ in this case.

If $\gamma$ is any homotopically non-trivial curve on $\chb$, then it is homotopic
to either a geodesic $\gamma'$ or a (multiple of a) cross-section $C$ of a cusp and
either \hbox{$i(\gamma)\ge i(\gamma')$} or \hbox{$i(\gamma)\ge i(C)$.} In either case,
\hbox{$i(\gamma)\ge 2\pi$.} (Recall that Lemma \ref{cuspcross} implies that \hbox{$i(C)=2\pi$.})
\end{proof}

As an immediate corollary of Lemmas \ref{intlowerbound}, \ref{finiteproj} and \ref{quasiapprox}
we see that every essential curve in a hyperbolic domain has length greater than
$2\pi$ in the Thurston metric.

\begin{corollary}
\label{Thurstonlengthbound}
If $\Omega\subset \rs$ is a hyperbolic domain and
$\alpha:S^1\to \chb$ is homotopically non-trivial in $\Omega$, then
$$l_\tau(\alpha)>2\pi.$$
\end{corollary}

\subsection{Intersection bounds for short simple closed geodesics in $\chb$}
\label{shortgeo}

We may use a variation on the construction from  Lemma \ref{thickintbound} 
to uniformly bound $i(\gamma)$ for all short geodesics.  

\begin{lemma}
\label{shortgeointbound}
Suppose that $\Omega$ is the complement of finitely many points in $\rs$. 
If $\gamma$  is a simple closed geodesic of length less than $2\sinh^{-1}(1)$, then
$$2\pi\le i(\gamma)  \leq 2\pi + 2\tan^{-1}(\sinh(l_h(\gamma)/2)) \leq \frac{5\pi}{2} $$
\label{short}
\end{lemma}

The lower bound $i(\gamma)\ge 2\pi$ follows immediately from 
Lemma \ref{intlowerbound}.
For the moment, we will give a much briefer proof that \hbox{$i(\gamma)\le 6\pi$} and defer the proof
of the better upper bound to the appendix.

\medskip\noindent
{\bf Proof that $i(\gamma)\le 6\pi$:}
Let $x$ be a point in $N(\gamma)$ with injectivity radius $\sinh^{-1}(1)$
and let
$C$ be the cross-section of $N(\gamma)$ which contains
$x$. Part (5) of the Collar Lemma implies that
$$\sinh(l_h(\gamma)/2)\cosh d_h(x,\gamma)=1,$$
and part (1) then implies that
$$l_h(C)= \frac{l_h(\gamma)}{\sinh(l_h(\gamma)/2)} \leq 2.$$
We may therefore divide  $\gamma'$ into
$\left\lceil \frac{l_h(C)}{G(\sinh^{-1}(1))}\right\rceil=3$ segments  of length
at most $G(\sinh^{-1}(1))\approx .83862$, denoted $C_1$, $C_2$ and $C_3$. 
Each $C_i$ is homotopic (relative to its endpoints) to a geodesic arc $\alpha_i$
and $i(C_i)=i(\alpha_i)$ for all $i$, since no edge can intersect either $\alpha_i$ or
$C_i$ twice.
Lemma \ref{thickintbound} implies that $i(\alpha_i)\le 2\pi$.
Therefore,
$$i(C) = i(C_1)+i(C_2)+i(C_3)=i(\alpha_1)+i(\alpha_2)+i(\alpha_3) \leq
3(2\pi)=   6\pi.$$ 
Lemma \ref{collarcross} implies that $i(\gamma)=i(C)$ completing the proof.
\eproof

\subsection{Angle bounds for edges and short closed geodesics}

The final tool we will need in the proof of Theorem \ref{projective} is
a bound on the angle between an edge of $\chb$ and a short simple closed
geodesic on $\chb$. 

\begin{lemma}
\label{anglebound}
Suppose that $\Omega$ is the complement of finitely many points in $\rs$. 
If $\gamma$ is a simple closed geodesic of length $l_h(\gamma)<2\sinh^{-1}(1)$,
then there is an edge $e_m$ intersecting $\gamma$ at an angle 
$$\phi_m \geq 
\sin^{-1}\left( {4\over 5}\right)\approx .9272.$$
Furthermore, any other edge $e_i$ intersecting $\gamma$ intersects in an angle
$$\phi_i  \geq 
\Phi = \sin^{-1}\left( {4\over 5\sqrt{2}+3}\right)\approx 0.4084.$$ 
\end{lemma}

\begin{proof}
Let $\gamma$ intersect edges $\{e_i\}_{i=1}^m$ at points $\{x_i\}_{i=1}^m$ and
with angles $\{\phi_i\}_{i=1}^m$. We may assume that the largest value of $\phi_i$
is achieved at $\phi_m$.
Furthermore, let $\theta^b_i$ be the exterior dihedral angle of the faces which meet at $e_i$
and let $\theta^c_i$ be  the exterior angle of $\gamma$ at $x_i$.  

A simple (Euclidean) calculation yields
$$\cos(\theta^c_i) = \sin^2(\phi_i)\cos(\theta^b_i)+\cos^2(\phi_i).$$

Consider $f_y:[0,{\pi\over 2}] \rightarrow [0,\pi]$, given by 
$$f_{y}(x) = \cos^{-1}(\sin^2(y)\cos(x)+\cos^2(y)).$$
Then $\theta^c_i = f_{\phi_i}(\theta^b_i)$. 
One may compute that
$$ \frac{df_{y}}{dx} = \frac{\sin^2(y)\sin(x)}{\sin(f_y(x))}>0\qquad {\rm and}
\qquad   \frac{d^2 f_y}{dx^2} = -\frac{\sin^2(y)\cos^2(y)(1-\cos(x))^2}{\sin^3(f_y(x))}< 0 $$
Since $f_y$ is an increasing concave down function, $f_y(0)=0$ and
$$\lim_{x\to 0^+}f_y'(x)  = \sin(y)$$ we see that
$$\theta^c_i \leq \sin(\phi_i)\ \theta^b_i.$$
Let $c(\gamma)$ be the geodesic curvature of $\gamma$. Then
$$c(\gamma)  = \sum \theta^c_i \leq \sin(\phi_m)\sum \theta^b_i = \sin(\phi_m)\  i(\gamma).$$ Lemma \ref{short}, and the fact that $2\pi\le c(\gamma)\le i(\gamma)$, imply that
$$ 2\pi \leq c(\gamma) \leq \sin(\phi_m)(2\pi + 2\tan^{-1}(\sinh(l_h(\gamma)/2)) \leq 
\frac{5\pi}{2}\sin(\phi_m).$$
Therefore,
$$\phi_m \geq \sin^{-1}\left(\frac{\pi}{\pi + \tan^{-1}(\sinh(l_h(\gamma)/2))}\right) \geq \sin^{-1}\left(\frac{4}{5}\right) \approx .9272.$$

We now use the bound on $\phi_m$ to bound the other angles of intersection.
Let $e_i$ be an edge, where $i<m$, intersecting $\gamma$.
Identify the universal cover of $\chb$ with $\Hp$ and let
$g$ be a geodesic in $\Hp$ covering $\gamma$. There is a lift
$h_m$ of $e_m$  which is a geodesic intersecting $g$ at point $y_1$ with angle $\phi_m$. Translating along $g$ we have another lift $h_m'$ of $e_m$ intersecting $g$ at point $y_2$ a distance $l_h(\gamma)$ from $y_1$. Let $h_i$ be a lift of $e_i$ whose intersection
$z$ with $g$, lies between $y_1$ and $y_2$. 
Without loss of generality, we may assume that 
\hbox{$d(z,y_1) \leq l_h(\gamma)/2$}.
As edges do not intersect,  $h_i$ lies between $h_m$ and $h_m'$.  
Let $p$  and $q$ be the endpoints of $h_m$ and 
let $T_p$ (respectively $T_q$) be  the triangle with vertices $y_1$, $z$,
and $p$ (respectively $y_1$, $z$ and $q$). 
Let $\psi_p$ and $\psi_q$ be  the internal angles of $T_p$ and $T_q$  at $z$.
We may assume that the internal angle of $T_p$ at $y_1$ is \hbox{$\pi-\phi_m$} and hence
that the internal angle of $T_q$ at $y_1$ is $\phi_m$.
By construction, 
$$\phi_i \geq \min\{\psi_p,\psi_q\}=\psi_p.$$

Basic calculations in  hyperbolic trigonometry (see \cite[Theorem 7.10.1]{beardon}), give
that
$$\cosh(d(z,y_1)) = \frac{1 -\cos{\psi_p}\cos{\phi_m}}{\sin{\psi_p}\sin{\phi_m}}$$
and
$$\sinh(d(z,y_1))  = \frac{\cos{\psi_p}-\cos{\phi_m}}{\sin{\psi_p}\sin{\phi_m}}.$$
One may then check that
$$\cosh(d(z,y_1))+\cos \phi_m \sinh(d(z,y_1))=\frac{\sin\phi_m}{\sin\psi_p}.$$
Therefore, 
$$\psi_p= \sin^{-1}\left( \frac{\sin(\phi_m)}{\cosh(d(z,y_1))+\cos(\phi_m)\sinh(d(z,y_1))}  \right) .$$
Since, \hbox{$d(z,y_1)\le l_h(\gamma)/2\le\sinh^{-1}(1)$} and \hbox{$\sin(\phi_m)\ge 4/5$,}
$$\phi_i \geq\psi_p\ge \Phi = \sin^{-1} \left( \frac{\frac{4}{5}}{\sqrt{2} +\frac{3}{5}} \right) \approx .4084.$$
\end{proof}

\medskip\noindent
{\bf Remark:} Using the elementary estimate $i(\gamma)\le 6\pi$ in Lemma \ref{anglebound} we obtain an angle bound of $\Phi \geq .141$. Using this we can prove a weaker form of Theorem \ref{projective} where $K_0$ is replaced by
$\hat K_0\approx 14.09$ and $K$ is replaced with $\hat K\approx 16.81$.

\section{The proof of Theorem \ref{projective}}

We are now prepared to prove Theorem \ref{projective} which asserts
that the nearest point retraction is a uniform quasi-isometry in the Thurston
metric. Our approximation result, Lemma \ref{quasiapprox}, allows us to
reduce to the finitely punctured case. We begin by producing an upper bound
on the distance between points in $\Omega$ whose images on $\chb$ are close.

\begin{prop}
\label{distbound}
Suppose that $\Omega$ is the complement of finitely many points in $\rs$. 
If $z,w \in \Omega$ and $d_{\chb}(r(z),r(w)) \le G(\sinh^{-1}(1))$ then  
$$d_{\tau}(z,w) \le K_0 =  G(\sinh^{-1}(1)) + 2\pi \approx 7.1219.$$
\end{prop}

\begin{proof}
Let $\alpha:[0,1]\to\chb$ be
a shortest path joining $r(z)$ to $r(w)$ and
let $G = G(\sinh^{-1}(1)) \approx .838682.$ By assumption,
$$l_{h}(\alpha) = d_{\chb}(r(z),r(w)) \le G.$$

If $\alpha$ contains a point of injectivity radius greater than $\sinh^{-1}(1)$, then
Lemma \ref{thickintbound} implies that \hbox{$i(\alpha) \le 2\pi$.} Lemma \ref{finiteptoh}
then gives that
$$d_\tau(z,w) \leq l_\tau(r^{-1}(\alpha)) = l_h(\alpha) + i(\alpha) \le G + 2\pi =K_0\approx 7.1219$$

If $\alpha$ does not contain a point of injectivity radius greater than $\sinh^{-1}(1)$,
then, by the Collar Lemma, $\alpha$ is either entirely contained in a cusp or 
in the collar $N(\gamma)$ about a simple closed geodesic  $\gamma$ with 
\hbox{$l_h(\gamma)<2\sinh^{-1}(1)$.}

Let $\alpha(1)$ lie on cross-section $C$ of a cusp or collar. 
Every cross-section of a cusp has length at most 2,  and we may argue, exactly as
in Section \ref{shortgeo}, that if $C$ is a cross-section of a collar, then
$$l_h(C) =\sinh({\rm inj}_{\chb}(\alpha(1))\left(  \frac{l_h(\gamma )}{\sinh(l_h(\gamma)/2)}\right) \leq 2.$$
In either case, \hbox{$l_h(C)\le 2.$} If $C$ is a cross-section of a cusp, then
\hbox{$i(C)=2\pi$,} by Lemma \ref{cuspcross}, while if $C$ is a cross-section of a collar,
then, by Lemmas \ref{collarcross} and \ref{shortgeointbound},  \hbox{$i(C) = i(\gamma) \leq 5\pi/2$.}
Therefore, Lemma \ref{finiteptoh} implies that
$$l_\tau(r^{-1}(C))=  l_h(C) + i(C) \leq 2 + \frac{5\pi}{2} \approx 9.8540.$$

If $\alpha$ is entirely contained in a cusp, then we may join $\alpha(0)$ to a point
$p$ on the  cross-section $C$  by an arc $B$ orthogonal to
each cross-section such that \hbox{$l_h(B)\le l_h(\alpha)\le G$.}
Also, since  each edge is orthogonal to every cross-section,
$B$ either misses every edge of $\chb$ or is a subarc of an edge. If $B$ misses every edge, then $i(B) = 0$ and  \hbox{$B' = r^{-1}(B)$} is an arc joining $z$ to a point $x$ on  $r^{-1}(C)$
of length \hbox{$l_\tau(B')=l_h(B)$.}
If $B$ is a subarc of an edge $e$, then $r^{-1}(B)$  contains a vertical (in the Euclidean coordinates described in Lemma \ref{finiteproj}) arc $B'$ which joins $z$ to a point $p$ on $r^{-1}(C)$
such that \hbox{$l_\tau(B')=l_h(B)$.} In both cases,  since \hbox{$l_\tau(r^{-1}(C))\le  2 + \frac{5\pi}{2}$,}
$p$ can be joined to $w\in r^{-1}(C)$ by an arc of length at most
$1 + \frac{5\pi}{4}$ in the Thurston metric. Therefore,
$$d_\tau(z,w) \leq d_\tau(z,p) + d_\tau(p,w) \leq l_\tau(B') + 
 1 + \frac{5\pi}{4}\leq  G +  1+ \frac{5\pi}{4} \approx 5.7657.$$

Finally, suppose that $\alpha$ is contained in the
collar $N(\gamma)$ about a simple closed geodesic
$\gamma$ with $l_h(\gamma)<2\sinh^{-1}(1)$.
We may assume that $\alpha(0)$ lies on
cross-section $S^1\times\{s\}$, $\alpha(1)$ lies on the cross-section $C$
given by \hbox{$S^1\times\{t\}$,} $s>0$, \hbox{$s>|t|$} and \hbox{$|s-t|\le G$.}
As every edge that intersects the collar $N(\gamma)$
must intersect the core geodesic $\gamma$ (see Lemma \ref{collarcross}),
we can decompose $N(\gamma)$ along edges to obtain a collection of quadrilaterals with two opposite sides corresponding to edges. As the edges intersect $\gamma$ in an angle at least
$\Phi$, we can foliate each quadrilateral by geodesics which also intersect $\gamma$ in an 
angle greater than $\Phi$. Thus, we obtain a foliation of $N(\gamma)$ by such geodesics. We let $B$ be the subarc of a leaf in the foliation joining $\alpha(0)$ to a point $p$ 
on the cross-section $C$.
Then, either $B$  does not intersect any edges of  $\chb$ or is a subset of an edge. 
In either case, $B$ belongs to geodesic $g$ which intersects
$\gamma$ at an angle \hbox{$\phi\ge\Phi$.}

We now use the angle bound to obtain an upper bound on the length of $B$, which will
allow us to complete the argument much as in the cusp case.
The hyperbolic law of sines implies that
$$l_h(B) = \sinh^{-1}\left(\frac{\sinh(s)}{\sin(\phi)}\right)- \sinh^{-1}\left(\frac{\sinh(t)}{\sin(\phi)}\right) $$
If we let \hbox{$f(s)=\sinh^{-1}\left(\frac{\sinh(s)}{\sin(\phi)}\right)$,} then
\hbox{$l_h(B)= f(s)-f(t).$} Since $f$ is odd  and increasing and \hbox{$|s-t|\le G$,}
$$l_h(B)\le \sup\left\{ f(s)-f(s-G)\ |\ s\ge { G\over 2}\right\}.$$
Since $f(s)-f(s-G)$ is decreasing on \hbox{$[{G\over 2},\infty)$} (since \hbox{$f''(t)\le 0$} for all $t$) it achieves its maximum value
when $s={G\over 2}$,
which implies that
$$l_h(B) \leq 2\sinh^{-1}\left(\frac{\sinh(G/2)}{\sin(\Phi)}\right) \approx 1.8831.$$

Since $B$ is either disjoint from all edges of  $\chb$ or is a subset of an edge, 
we argue as before to show that $r^{-1}(B)$ contains an arc $B'$ joining $z$ to a point $p$
in $r^{-1}(C)$ such that   
$$l_\tau(B') = l_h(B) \leq  2\sinh^{-1}\left(\frac{\sinh(G/2)}{\sin(\Phi)}\right) \approx 1.8831.$$
Therefore, as before,
\begin{eqnarray*}
d_\tau(z,w) & \leq &  d_\tau(z,p) + d_\tau(p,w) \\
& \leq & l_\tau(B') + 1 +\frac{5\pi}{4}\\
 & \leq &  
 2\sinh^{-1}\left(\frac{\sinh(G/2)}{\sin(\Phi)}\right)  +1+ \frac{5\pi}{4}\approx 6.8101\\
\end{eqnarray*}

Therefore, we see that in all cases
$$d_{\tau}(z,w)   \le K_0 \approx 7.1219.$$
\end{proof}

As a nearly immediate corollary, we show that the nearest point retraction
is a uniform quasi-isometry in the finitely punctured case.

\begin{corollary}
\label{finiteqi}
Suppose that $\Omega$ is the complement of finitely many points in $\rs$.  Then,
$$ d_{\chb}(r(z),r(w)) \leq d_{\tau}(z,w) \leq K\ d_{\chb}(r(z),r(w))+K_0$$
where   $K_0  \approx 7.12$ and
$K = \frac{K_0}{G(\sinh^{-1}(1))}\approx 8.49.$
\end{corollary}

\begin{proof}
The lower bound follows from the fact that $r$ is 1-Lipschitz in
the Thurston metric (see Lemma \ref{finiteptoh}). To prove the upper bound we consider
a shortest path $\alpha$ in $\chb$ from $r(z)$ to $r(w)$.  We decompose $\alpha$
into \hbox{$\left\lceil\frac{d_{\chb}(r(z),r(w))}{G(\sinh^{-1}(1))}\right\rceil$}
geodesic segments of length at most \hbox{$G(\sinh^{-1}(1))$.}
Then, applying Lemma \ref{distbound},
$$d_{\tau}(z,w)  \leq K_0 \left\lceil\frac{d_{\chb}(r(z),r(w))}{G(\sinh^{-1}(1))}\right\rceil \leq K d_{\chb}(r(z),r(w)) + K_0$$
where $K = \frac{K_0}{G(\sinh^{-1}(1))}$.
\end{proof}

We  may now combine Corollary \ref{finiteqi} and Lemma \ref{quasiapprox}
to prove Theorem \ref{projective} in the general case.

\medskip\noindent
{\bf Theorem \ref{projective}.}
{\em
Let $\Omega$ be a hyperbolic domain. Then the nearest point
retraction  $r:\Omega \rightarrow \chb$  is a $(K,K_0)$-quasi-isometry
with respect to the Thurston 
metric $\tau$ on $\Omega$ and the  intrinsic hyperbolic metric on $\chb$
where $K\approx 8.49$ and $K_0\approx 7.12$.
In particular,
$$ d_{\chb}(r(z),r(w)) \leq d_{\tau}(z,w) \leq K\ d_{\chb}(r(z),r(w))+K_0.$$
Furthermore, if $\Omega$ is simply connected,  then
$$d_{\chb}(r(z),r(w)) \leq d_{\tau}(z,w) \leq K'\ d_{\chb}(r(z),r(w))+K'_0.$$
where $K'\approx 4.56$ and $K'_0 \approx 8.05$.
}
\begin{proof}
Let $\{ X_n\}$ be a  nested finite approximation for $\rs-\Omega$,
let \hbox{$\Omega_n=\rs-X_n$} with projective metric $\tau_n$ 
and let $r_n:\Omega_n\to {\rm Dome}(\Omega_n)$ be
the associated nearest point retraction. Then, by Lemma \ref{quasiapprox},
\hbox{$\{d_{{\rm Dome}(\Omega_n})(r_n(z),r_n(w))\}$} converges to 
$d_{\chb}(r(z),r(w))$  and
\hbox{$\{d_{\tau_n}(z,w)\}$} converges to $d_\tau(z,w)$ for all \hbox{$z,w\in\Omega$.}
Corollary \ref{finiteqi} implies that
$$ d_{{\rm Dome}(\Omega_n)}(r_n(z),r_n(w)) \leq d_{\tau_n}(z,w) \leq K\ d_{{\rm Dome}(\Omega_n)}(r_n(z),r_n(w))+K_0$$ 
for all $n$ and all  $z,w\in\Omega$,
so the result follows.

Now suppose that is $\Omega$ simply connected and let $k_n$ be the minimum injectivity radius in $X_n$ along the shortest curve $\alpha_n$ joining $r_n(z)$ and $r_n(w)$.
Lemma \ref{quasiapprox} implies that  $k_n \rightarrow \infty$ as $n \rightarrow \infty$. 
As in the proof of Corollary \ref{finiteqi}, we decompose $\alpha_n$ into segments of length $G(k_n)$
and conclude that
$$d_{\tau_n}(z,w) \leq \left(\frac{2\pi+G(k_n)}{G(k_n)}\right) d_{{\rm Dome}(\Omega_n)}(r_n(z),r_n(w))+(2\pi+G(k_n))$$
As $\lim_{n\rightarrow\infty} G(k_n)  = 2\sinh^{-1}(1) $, we again apply Lemma \ref{quasiapprox}
to obtain
$$d_{\tau}(z,w) \leq K'  d_{{\rm Dome}(\Omega)}(r(z),r(w))+K'_0$$
where $K' = \frac{2\pi+2\sinh^{-1}(1)}{2\sinh^{-1}(1)} \approx 4.56$ and $K'_0 =  2\pi+2\sinh^{-1}(1) \approx 8.05$. 
\end{proof}

\section{Consequences of Theorem \ref{projective}}

In this section, we derive  a series of corollaries of Theorem \ref{projective}.
We begin by combining Theorem \ref{projective} with Beardon and Pommerenke's
work \cite{BP} and Theorem \ref{projtoqh} to obtain explicit quasi-isometry bounds for
the nearest point retraction with respect to the Poincar\'e metric on
a uniformly perfect domain.

\medskip\noindent
{\bf Corollary \ref{poincquant}.}
{\em Let $\Omega$ be a uniformly perfect domain in $\rs$ and
let $\nu>0$ be a lower bound for its injectivity radius in the Poincar\'e metric,
then
$${1\over 2\sqrt{2}(k+{\pi^2\over 2\nu})}\ d_{\chb}(r(z),r(w)) \leq d_{\rho}(z,w) \leq K\ d_{\chb}(r(z),r(w))+K_0$$
for all $z,w\in\Omega$ where $k=4+\log (2+\sqrt{2})$.
}

\begin{proof}
Combining Corollary \ref{poincandqh}, Theorem \ref{projtoqh}  and
the fact that \hbox{$\rho(z)\le \tau(z)$,} we see that 
$${1\over 2 \sqrt{2}(k+{\pi^2\over 2\nu})}\tau(z)\le  \rho(z)\le \tau(z)$$
for all $z\in\Omega$. It follows that
$${1\over 2 \sqrt{2}(k+{\pi^2\over 2\nu})}\ d_\tau(z,w)\le  d_\rho(z,w)\le d_\tau(z,w)$$
for all $z,w\in\Omega$. The result then follows by applying Theorem
\ref{projective}.
\end{proof}

If $\Omega$ is simply connected, then we may apply Theorem \ref{scThurPoin}
to obtain better quasi-isometry constants. 
Bishop (see Lemma 8 in \cite{bishop-bowen}) first showed that the nearest point
retraction is a quasi-isometry in the simply connected case, but did not obtain
explicit constants. 

\medskip\noindent
{\bf Corollary \ref{qiinsc}.}
{\em
If $\Omega$ is a simply connected domain in $\rs$, then
$${1\over 2}\ d_{\chb}(r(z),r(w)) \leq d_{\rho}(z,w) \leq K'\ d_{\chb}(r(z),r(w))+K'_0$$
for all $z,w\in\Omega$.
}

\medskip

We now establish Marden and Markovic's conjecture that the nearest point
retraction is Lipschitz (with respect to the Poincar\'e metric on $\Omega$) if and
only if $\Omega$ is uniformly perfect.

\medskip\noindent
{\bf Corollary \ref{MMconj}.}
{\em
If $\Omega$ is a hyperbolic domain, then
the nearest point retraction is Lipschitz (where we give $\Omega$ the Poincar\'e metric)
if and only if $\Omega$ is uniformly perfect. Moreover, if $\Omega$ is not uniformly
perfect, then $r$ is not a quasi-isometry with respect to the Poincar\'e metric.}

\medskip

\begin{proof}  
If $\Omega$ is uniformly perfect, then Corollary \ref{poincquant} implies that
$r$ is Lipschitz.

If $\Omega$ is not uniformly perfect,
then there exist essential  round annuli  in $\Omega$
with arbitrarily large moduli.
For each $n$, let $A_n$ be a separating round annulus in $\Omega$ of modulus $\frac{4n}{2\pi}$.
We may normalize so that $\infty\notin\Omega$ and \hbox{$A_n=\{ z\ |\ e^{-2n} \le |z| \le e^{2n}\}$}
and choose \hbox{$x_n=e^{-n}$} and \hbox{$y_n=e^{n}$.} 
If $z$ lies in the subannulus \hbox{$B_n=\{ z\ |\ e^{-n} \le |z| \leq e^{n}\}$,}
then
$$|z| -e^{-2n}\leq \delta(z) \leq |z| + e^{-2n},$$
so 
$$\frac{1}{2} |z|\leq (1-e^{-n}) |z|  \leq \delta(z)\leq (1+e^{-n})|z|\leq 2|z|.$$
As every path from $x_n$ to $y_n$ contains an arc joining the boundary components of $B_n$, 
this implies that 
$$n  \le d_{q}(x_n,y_n) \le 4n.$$
Therefore, $d_q(x_n,y_n)\to\infty$, so, by Theorem \ref{projtoqh}, \hbox{$d_\tau(x_n,y_n)\to\infty$.}
Theorem \ref{projective} then implies that 
$$d_h(r(x_n),r(y_n))\to \infty.$$ 

On the other hand, if $\gamma$ is the segment of the
real line joining $x_n$ to $y_n$, then \hbox{$\beta(z)>n$} for all $z\in\gamma$,
so by Theorem \ref{BPresult}, 
$$l_\rho(\gamma) \leq \left(\frac{2k+\frac{\pi}{2}}{k+n}\right) l_q(\gamma) \leq \left(\frac{2k+\frac{\pi}{2}}{k+n}\right) 4n \leq C$$ for all $n$ and some constant $C > 0$. 
Therefore, 
$$d_\rho(x_n,y_n)\le C$$ for all $n$. 
It follows that $r$ is not Lipschitz
with respect to the Poincar\'e metric on $\Omega$. Moreover, $r$ is not even a
quasi-isometry.
\end{proof}

\section{The lift of the nearest point retraction}
\label{thelift}

In this section we show that if $\Omega$ is uniformly perfect, then the nearest point retraction lifts to a quasi-isometry, with quasi-isometry constants depending only on the geometry of the domain. The proof requires only 
Lemma \ref{thickintbound} and does not make use of the analysis of the thin part obtained
in Lemmas \ref{shortgeointbound} and \ref{anglebound}. We then use this result to
show that $r$ is homotopic to a quasiconformal homeomorphism if $\Omega$ is 
uniformly perfect.

\medskip\noindent
{\bf Theorem \ref{nprlift}.}
{\em Suppose that $\Omega$ is a uniformly perfect hyperbolic domain and
$\nu>0$ is a lower bound for its injectivity radius in the Poincar\'e metric.
Then the nearest point retraction lifts to a quasi-isometry
$$\tilde r:\tilde\Omega \to {\rm Dome}(\tilde\Omega)$$
between the universal
cover of $\Omega$ (with the Poincar\'e metric) and the universal cover
of $\chb$ where the quasi-isometry constants depend only on $\nu$.
In particular,
$${1\over 2\sqrt{2}(k+{\pi^2\over 2 \nu})}\ d_{{\rm Dome}(\tilde\Omega)}(\tilde r(z),\tilde r(w)) \leq d_{\tilde\Omega}(z,w)\le L(\nu) d_{{\rm Dome}(\tilde\Omega)}(\tilde r(z),\tilde r(w)) + L_0(\nu)$$
for all $z,w\in \tilde\Omega$, where
 \hbox{$m=\cosh^{-1}(e^2)\approx 2.69$},
$$g(\nu)={1\over 2}e^{-m}e^{-\pi^2\over 2\nu},\qquad
L(\nu)={(2\pi+G(g(\nu))\over G(g(\nu))},\  and$$
$$L_0(\nu)=(2\pi+G(g(\nu)))\le L_0=2\pi+2\sinh^{-1}(1)\approx 8.05.$$
}

\begin{proof}
Corollary \ref{poincquant} implies that $r$ is $(2\sqrt{2}(k+{\pi^2\over 2 \nu}))$-Lipschitz,
so therefore $\tilde r$ is also $(2\sqrt{2}(k+{\pi^2\over 2 \nu}))$-Lipschitz. 

Given \hbox{$z,w\in\tilde \Omega$}, let \hbox{$\tilde\beta:[0,1]\to \tilde\Omega$}
be the geodesic
arc in $\tilde \Omega$ joining them and let \hbox{$\beta:[0,1]\to \Omega$} be its projection
to $\Omega$ (i.e. \hbox{$\beta=\pi_\Omega\circ\tilde\beta$} where
\hbox{$\pi_\Omega:\tilde\Omega\to\Omega$} is the universal covering map.)
Then \hbox{$d_{{\rm Dome}(\tilde\Omega)}(\tilde r(z),\tilde r(w))$} is the length of
the geodesic arc $\alpha$ on $\chb$ joining $r(\beta(0))$ and $r(\beta(0))$
in the proper homotopy class of $r\circ \beta$.

Let $\{X_n\}$ be a nested finite approximation to $\rs-\Omega$,
let \hbox{$\Omega_n = \rs-X_n$}, let $\tau_n$ be the Thurston metric on $\Omega_n$,
and let \hbox{$r_n:\Omega_n\to {\rm Dome}(\Omega_n)$} be
the nearest point retraction. Let $\alpha_n$ be the geodesic arc in 
${\rm Dome}(\Omega_n)$ joining 
$r_n(\beta(0))$ and $r_n(\beta(1))$ in the proper homotopy class of 
$r_n \circ \beta$.
Let $k_n$ be the minimum of the injectivity radius of ${\rm Dome}(\Omega_n)$ at
points in $\alpha_n$. If \hbox{$a_n\le k_n$}, then
we may divide $\alpha_n$ into segments of length $G(a_n)$ and apply
Lemmas \ref{finiteptoh} and \ref{thickintbound}, as in the proof of Corollary \ref{finiteqi},
to obtain the bound,
\begin{eqnarray*}{}
l_\rho(r_n^{-1}(\alpha_n))  & \leq & \left(2\pi + G(a_n))\right)\left\lceil\frac{l_h(\alpha_n)}{G(\nu_n)}\right\rceil \\
 & \leq & {(2\pi+G(a_n)\over G(a_n)} l_h(\alpha_n) + (2\pi+G(a_n)).\\
 \end{eqnarray*}

Since $\{{\rm Dome}(\Omega_n)\}$ converges to $\chb$, $\{ r_n\}$ converges to $r$,
\hbox{$l_h(\alpha_n)\le l_{\tau_n}(\beta)$} for all $n$, and $\{\tau_n\}$ converges to $\tau$
(see Lemma \ref{quasiapprox}), $\{\alpha_n\}$ converges, up to subsequence, to 
a geodesic path $\alpha_\infty$ on $\chb$ joining $r(\beta(0))$ to $r(\beta(1))$.
Since $r_n(\alpha_\infty)$
lies arbitrarily close to $\alpha_n$ (for large enough $n$ in the subsequence)
which is properly homotopic to $r_n\circ\beta$,
we see that $\alpha_\infty$ is homotopic to $r\circ\beta$, and hence agrees with $\alpha$.

Lemma 9.1 in \cite{BC03} gives that
if $\nu>0$ is a lower bound for the injectivity radius of $\Omega$, in
its Poincar\'e metric, then $g(\nu)$
is a lower bound for the injectivity radius of $\chb$. Since $\{\alpha_n\}$ converges to
$\alpha$, up to subsequence, 
\hbox{$\lim{\rm inj}_{{\rm Dome}(\Omega_n)}(z)={\rm inj}_{\chb}(z)$} for all $z\in \Omega$ (see Lemma \ref{quasiapprox} again), and the injectivity radius function
is 1-Lipschitz on any hyperbolic surface,
we see that \hbox{$\liminf k_n\ge g(\nu)$}. So, we may choose $a_n$ so that 
\hbox{$\lim a_n=g(\nu)$},
Therefore, by taking limits, we see that
$$l_\rho(r^{-1}(\alpha))  \leq {(2\pi+G(g(\nu))\over G(g(\nu))} l_h(\alpha) +
 (2\pi+G(g(\nu))).$$

The curve $r^{-1}(\alpha)$ contains an arc joining $r(\beta(0))$ to $r(\beta(1))$
in the homotopy class of $\beta$, so
$$l_\rho(r^{-1}(\alpha))\ge l_\rho(\beta)=d_{\tilde\Omega}(z,w).$$
Therefore,
$$d_{\tilde\Omega}(z,w)\le {(2\pi+G(g(\nu))\over G(g(\nu))} d_{{\rm Dome}(\tilde\Omega)}(\tilde r(z),\tilde r(w)) + (2\pi+G(g(\nu))).$$
\end{proof}

\medskip\noindent
{\bf Remark:} The proof of Corollary \ref{MMconj} also implies that $\tilde r$ is not
a quasi-isometry with respect to the Poincar\'e metric if $\Omega$ is not uniformly
perfect.  Since any homotopically
non-trivial curve in $\Omega$ has length greater than $2\pi$ in the Thurston metric
(by Lemma \ref{Thurstonlengthbound}), $\tilde r$ is not even a quasi-isometry
with respect to the (lift of  the) Thurston metric on $\tilde\Omega$ if $\Omega$ is not
uniformly perfect.

\medskip

We now use the Douady-Earle extension theorem to obtain a quasiconformal
map homotopic to the nearest point retraction whenever $\Omega$ is uniformly perfect.
The constants obtained in this proof are explicit, but clearly far from optimal.

\medskip\noindent
{\bf Corollary \ref{upqc}.}
{\em
If $\Omega$ is uniformly perfect and $\nu>0$ is a lower bound for its
injectivity radius in the Poincar\'e metric, then there is a conformally natural
$M(\nu)$-quasiconformal map $\phi:\Omega\to\chb$ which admits a
bounded homotopy to $r$, where
$$M(\nu)=4(10)^8e^{70N(\max\{2\sqrt{2}(k+{\pi^2\over 2 \nu}),L(\nu)\},L_0)}$$
and $N(K,C)= e^{1546K^4\max\{C,1\} }$.
Moreover, if $\Omega$ is not uniformly perfect, then there does
not exist a bounded homotopy of $r$ to a quasiconformal map.
}

\begin{proof}
A close examination of the standard proof that quasi-isometries  of $\Hp$ extend to
quasisymmetries of \hbox{$S^1=\partial_\infty\Hp$}, see, e.g.,  Lemma 3.43 of \cite{kapovich} and Lemma 5.9.4 of  \cite{ThBook}, yields that a \hbox{$(K,C)$}-quasi-isometry of $\Hp$ extends to a \hbox{$N(K,C)$}-quasisymmetry
of $S^1$.
The Beurling-Ahlfors extension of a $k$-quasisymmetry of $S^1$ is a $2k$-quasiconformal
map (see \cite{lehtinen}). Proposition 7 of Douady-Earle \cite{douady-earle} then implies
that every $k$-quasisymmetry admits a conformally natural  quasiconformal extension
with dilatation at most \hbox{$4(10)^8e^{70k}$.}

First suppose that $\Omega$ is uniformly perfect and
identify both $\tilde\Omega$ and ${\rm Dome}(\tilde\Omega)$ with $\Hp$.
Then $\tilde r$ lifts to a \hbox{$(\max\{J(\nu),L(\nu)\},L_0)$}-quasi-isometry  of $\tilde \Omega$
and so extends to
a \hbox{$N(\max\{J(\nu),L(\nu)\},L_0)$}-quasisymmetry of $S^1$, which itself extends to a conformally natural $M(\nu)$-quasiconformal map \hbox{$\tilde g:\tilde\Omega\to{\rm Dome}(\tilde\Omega)$}. It follows from Proposition 4.3.1 and Theorem 4.3.2 in
Fletcher-Markovic \cite{FM} that the straight-line homotopy between $\tilde r$ and
$\tilde g$ is bounded.
Since $\tilde r$ is the lift of a conformally natural map and the Douady-Earle extension
is conformally natural, we see that $\tilde g$ descends to a conformally natural
$M(\nu)$-quasiconformal map \hbox{$g:\Omega\to\chb$} and that the straight-line
homotopy between $\tilde r$ and $\tilde g$ descends to a bounded homotopy.

A quasiconformal homeomorphism
between hyperbolic surfaces is a quasi-isometry (see \cite[Theorem 4.3.2]{FM}). Thus, if $r$ admits a bounded
homotopy to a quasiconformal map, it must itself be a quasi-isometry. Therefore,
by Corollary \ref{MMconj}, if $r$ admits a bounded homotopy to a quasiconformal
map, $\Omega$ must be uniformly perfect.

\end{proof}

\section{Round annuli}
\label{round}

In this section, we will consider the special case of round annuli. A hyperbolic domain
contains points with small injectivity radius (in the Poincar\'e metric)
if and only if it contains  round annuli with
large modulus, so round annuli are natural test cases for the constants we obtain. Moreover,
hyperbolic domains which are not uniformly perfect contain round annuli of
arbitrarily large modulus.

Let $\Omega(s)$ denote the round annulus lying
between concentric circles of radius $1$ and $e^s>1$ about the origin. One may
calculate that in the Poincar\'e metric  $\Omega(s)$ is
a complete hyperbolic annulus with core curve of length ${2\pi^2\over s}$ and that
${\rm Dome}(\Omega(s))$ is a complete hyperbolic annulus with core curve of
length $2\pi\over \sinh({s\over 2})$ (see Example 3.A in Herron-Minda-Ma \cite{HMM05}
and Theorem 2.16.1 in Epstein-Marden \cite{EM87}).
More explicitly, $\Omega(s)$ is homeomorphic
to $S^1\times \mathbb{R}$ with the metric
$$ds^2=\left({\pi\over s}\right)^2\cosh^2(t)d\theta^2+dt^2$$ and
${\rm Dome}(\Omega(s))$ is homeomorphic
to $S^1\times \mathbb{R}$ with the metric
$$ds^2=\left({1\over \sinh({s\over 2})}\right)^2\cosh^2(t)d\theta^2+dt^2.$$
The nearest point retraction
$r_s:\Omega(s)\to {\rm Dome}(\Omega(s))$ takes the core curve of $\Omega(s)$ to the core
curve of ${\rm Dome}(\Omega(s))$. More generally, it takes any cross-section of
$\Omega(s)$ to a cross-section of ${\rm Dome}(\Omega(s))$ (and restricts to a dilation)
and takes
any vertical geodesic perpendicular to the core curve of $\Omega(s)$ to
a vertical geodesic perpendicular to the core curve of ${\rm Dome}(\Omega(s))$.

Let \hbox{$\nu(s)={\pi^2\over s}$} be the minimal value  for the injectivity radius on $\Omega(s)$.
One may check that there exists a constant $C>0$
such that if \hbox{${\rm inj}_{\Omega(s)}(x)=\sinh^{-1}(1)$}, then 
\hbox{${\rm inj}_{{\rm Dome}(\Omega(s))}(r_s(x))\ge C$.}
We assume from now on that \hbox{$\nu(s)<\sinh^{-1}(1)$}. 
Let $$t_s=\cosh^{-1}\left({1\over\sinh^{-1}({\pi^2\over s})}\right)=O(\log(s))$$
as $s\to\infty$.
Choose points 
$x_1=(\theta_0,t_s) $ and $x_2=(\theta_0,-t_s)$ in $\Omega(s)$
for fixed $\theta_0\in S^1$. Notice that 
\hbox{${\rm inj}_{\Omega(s)}(x_1)= {\rm inj}_{\Omega(s)}(x_2)=\sinh^{-1}(1)$} and
$x_1$ and $x_2$ lie on opposite sides of a geodesic perpendicular to the core curve of
$\Omega(s)$. 
Therefore, $r_s(x_1)$ and $r_s(x_2)$ lie on opposite sites of a geodesic perpendicular
to the core curve of ${\rm Dome}(\Omega(s))$. Since, 
\hbox{${\rm inj}_{{\rm Dome}(\Omega(s))}(r(x_i))\ge C$,} we see that 
$$d(r_s(x_1),r_s(x_2))\ge 2\cosh^{-1}\left({\sinh(C)\over\sinh^{-1}({\pi \over \sinh({s\over 2})})}\right)=
O(s)$$ as $s\to \infty.$
It follows that the Lipschitz constant of $r_s$ is at least
$$ {\cosh^{-1}\left({\sinh(C)\over\sinh^{-1}({\pi \over \sinh({s\over 2})})}\right)\over
\cosh^{-1}\left({1\over\sinh^{-1}({\pi^2 \over s})}\right)}=O\left({s\over \log s}\right)$$
as $s\to \infty$. But, since $\nu(s)={\pi^2\over s}$, this expression is
also \hbox{$O({1\over|\log \nu(s)|\ \nu(s)})$} as $\nu(s)\to 0$. Notice that the Lipschitz constant
in Corollary \ref{poincquant} is 
\hbox{$2\sqrt{2}(k+{\pi^2\over \nu})=O(1/\nu)$} as $\nu\to 0$,
suggesting that our constants are close to having the correct form.
(We remark that Corollary 1.2 contradicts part (3) of Theorem 2.16.1 in \cite{EM87}, which
is incorrect as stated.)

If the lift of the nearest point retraction to a map
from $\tilde \Omega(s)$ to ${\rm Dome}(\tilde\Omega(s))$  is a
\hbox{$(K(s),C(s))$}-quasi-isometry, then
$$K(s)\ge  {\pi\sinh({s\over 2}) \over s}=O\left({e^{s\over 2}\over s}\right)$$
as $s\to\infty$, since it takes the lift of the core curve of $\Omega(s)$ to the lift of
the core curve  of ${\rm Dome}(\Omega(s))$ equivariantly. 
Notice that this expression is also \hbox{$O(\nu(s) e^{\pi^2\over2\nu(s)})$} as $\nu(s)\to 0$,
and that \hbox{$L(\nu)=O(e^{\pi^2\over 2\nu})$} in Theorem \ref{nprlift}, so again our constants
are close to having the right form. This also illustrates the fact that the quasi-isometry
constants of the lift will typically be larger than the quasi-isometry constants of
the original map.

We now consider the Thurston metric on the round annulus $\Omega(s)$. In
this case, the nearest point retraction is a homeomorphism. One may extend the
analysis in section \ref{fingeom} to show that the nearest point retraction
is an isometry, with respect to the Thurston metric, on lines through the origin.
On a circle about the origin, with length $L$ in the Thurston metric,
the nearest point retraction is a dilation with dilation constant ${L-2\pi \over L}$.
In particular, each such circle has length greater than $2\pi$ and the core curve of
$\Omega(s)$ has length \hbox{$2\pi+{2\pi\over \sinh({s\over 2})}$} in the Thurston metric.
(One may also derive these facts more concretely using Lemma \ref{projhoro} to 
compute the Thurston metric. See Example 3.A in \cite{HMM05} for an explicit calculation.)
It follows that the nearest point retraction
is a \hbox{$(1,2\pi)$}-quasi-isometry for all $\Omega(s)$. It is natural to ask what the
best possible general quasi-isometry constants are for the nearest point retraction.

\section{Appendix: Proofs of intersection number estimates}

In this appendix we give the proofs of Lemmas \ref{thickintbound} and \ref{shortgeointbound}.
The proof of Lemma \ref{thickintbound} follows the same outline as the proof
of Lemma 4.3 in \cite{BC03}, but is technically much simpler as we are working
in the setting of finitely punctured domains. We give the proof of Lemma \ref{thickintbound}
both to illustrate the simplifications achieved and to motivate the proof of Lemma
\ref{shortgeointbound} which has many of the same elements. 

Throughout this appendix, $X$ will be a finite set of points in $\rs$, $\Omega = \rs - X$ and $\chb = \partial CH(X)$.

\subsection{Families of support planes}
We begin by associating  a family of support planes
to a geodesic arc $\alpha$ in $\chb$. If these support planes all intersect, we will obtain
a bound on $i(\alpha)$. The proof of this bound is much easier in our simpler
setting than in \cite{BC03}. Our lemma is essentially a special case of Lemma 4.1 in \cite{BC03}.

Let $\alpha:[0,1]\to \chb$ be a geodesic arc transverse to the edges of $\chb$.
Let \hbox{$\{x_i = \alpha(t_i)\}_{i=1}^{n-1}$} be the finite collection of intersection points
of the interior $\alpha((0,1) )$ of $\alpha$ with edges of $\chb$.  
If  the initial point of $\alpha$ lies on an edge, we denote the point by $x_0$ and the edge
by $e_0$,
while if the endpoint of $\alpha$ lies on an edge, we denote the point $x_n$ and the edge
by $e_n$. Suppose that $x_i$ lies on edge $e_i$ and
that $e_i$ has exterior dihedral angle $\theta_i$. We let $F_i$ be
the face containing $\alpha(t_{i-1},t_i)$. (If $\alpha(0)$ lies on an edge $e_0$,
we let $F_0$ be the face abutting $e_0$ which does not contain $\alpha(0,t_1)$.)
At each point $x_i$, there is a 1-parameter family of support planes
\hbox{$\{P^i_s\ |\  0 \leq s \leq \theta_i\}$} where $P^{i}_{0}$ contains $F_{i}$
and the dihedral angle between $P^i_s$ and $P^i_t$ is \hbox{$|s-t|$.}
Concatenating these one parameter families we obtain the one parameter family of support planes
\hbox{$\{P_t\ |\ 0 \leq t \leq  i(\alpha)\}$} along $\alpha$, called the 
{\em full family of support planes} to $\alpha$.

In general, we will allow our families of support planes to omit sub-families
of support planes associated to the endpoints.
We say that 
$$\{ P_t\ |\ a\le t\le b\} \subset \{ P_t\ |\ 0\le t\le i(\alpha)\}$$
is a {\em family of support planes} to $\alpha$  if
\begin{enumerate}
\item
either $a=0$ or $\alpha(0)$ lies on
an edge $e_0$ and $a\le\theta_0$, and
\item
either $b=i(\alpha)$ or 
$\alpha(1)$ lies on an edge $e_n$ and $i(\alpha)-b\le \theta_n$.
\end{enumerate}
Notice that
a full family of support planes is a family of support planes in this definition.

\begin{lemma}
\label{sptbnd}
Suppose that $\Omega$ is the complement of a finite
collection $X$ of points in $\rs$ and $\alpha:[0,1] \rightarrow \chb$ is a  geodesic path
(which may  have endpoints on edges).
If $\{P_t\ | \ a\le t\le b\}$ is a family  of support planes along $\alpha$
such that the associated half-spaces satisfy $H_t\cap H_a\ne\emptyset$ for all $t\in[a,b]$,  then
$$ |b-a|  \leq \psi < \pi$$
where $\psi$ is the exterior dihedral angle between $P_a$ and $P_b$.
Moreover, $\alpha$ intersects each edge at most once.
\end{lemma}

{\bf Proof:}
Let $\{x_i\}$ be the intersection points of $\alpha$ with edges $\{e_i\}$. 
Let \hbox{$\{P_t\ |\ a_i \leq t \leq b_i\}$} be the support planes to $e_i$
in the family of support planes $\{P_t\ | \ a\le t\le b\}$. As \hbox{$H_t \cap H_a \neq \emptyset$,} then $g_t=P_t\cap P_a$ is a geodesic  if $P_t \neq P_a$.
If $e_i\subset P_a$, then $g_t=e_i$ for all $t\in [a_i,b_i]$ such that $P_t\ne P_a$.
If $e_i$ is not on $P_a$ then \hbox{$\{g_t|\ a_i \leq t \leq b_i\}$} is a continuously varying family of
disjoint geodesics lying between $g_{a_i}$ and $g_{b_i}$ on $P_a$. 

If every edge $\alpha$ intersects is contained in $P_a$, then the result is obvious, so we will assume
that there is an edge which $\alpha$ intersects which does not lie in $P_a$.
Choose $d$ so that $e_{d+1}$
is the first edge $\alpha$ intersects that does not belong to $P_a$.  

If an edge $e_j$ with $j>d$ lies in $P_a$ choose $e_j$
to be  the first edge such that $j> d$ and $e_j$ is on $P_a$.  Otherwise, choose $e_{j-1}$ to be the
last edge which $\alpha$ intersects.
By construction, \hbox{$j>d+1$,} since $x_{d+1}$ does not lie on $P_a$.
Notice that \hbox{$g_t = e_d$} for \hbox{$a_d < t \leq b_d$} and that
\hbox{$\{g_t\ |\  b_d\leq  t \leq b_{j-1}\} $}
is a continuous family of geodesics on $P_a$. 
If the family of geodesics
\hbox{$\{g_t \ |\  b_d\leq t \leq b_{j-1}\}$} are not all disjoint, then there 
must exist $i>d$, such that if \hbox{$t\in (a_{i+1},b_{i+1}]$,} then $g_t$ lies on the same
side of $g_{a_{i+1}}$ as $g_{a_{i}}$. However, this implies that  there exists
\hbox{$u\in (a_{i+1},b_{i+1}]$} such that $P_u$ intersects the interior of the face of $\chb$
contained in $P_{a_i}$. Since $P_u$ bounds a half-space  $H_u$ disjoint
from $CH(X)$, this is a contradiction. Therefore, \hbox{$\{g_t \ |\  b_d\leq t \leq b_{j-1}\}$} is a disjoint family of
geodesics.

If $e_j$ lies in $P_a$, then
there is a support plane $P_t$ separating some point on  $e_d$ from a point on $e_j$.  This is a contradiction. Therefore no edge $e_j$ with \hbox{$j>d$} lies in $P_a$ and \hbox{$b_{j-1}=b$,} so
\hbox{$\{g_t \ |\  b_d\leq t \leq b\}$} is a disjoint family of geodesics on $P_a$.
It follows that $\alpha$ intersects each edge exactly once, since otherwise there would be
distinct values  \hbox{$s,t\in[b_d,b]$} such that \hbox{$P_s=P_t$} and hence \hbox{$g_s=g_t$.}

Let $\psi_i$ be the exterior dihedral angle between $P_a$ and $P_{a_i}$ 
let \hbox{$\theta_i=b_i-a_{i}$} be the dihedral angle between 
$P_{b_i}$ and $P_{a_i}$ for all \hbox{$i> d$.}
Let $m$ be the maximal value of $i$ for an edge $e_i$
and let \hbox{$\psi_{m+1}=\psi$} be the exterior dihedral angle between $P_a$ and $P_b$
and let \hbox{$\theta_m=b_m-a_m$} denote the dihedral angle between $P_{a_m}$ and $P_{b}$.

We let $P_k = P_{a_k}$  for $k \leq m$ and $P_{m+1} = P_b$.
Now for \hbox{$d+1 \leq k \leq m$} we let $P$
be the unique plane perpendicular to the planes $P_a$, $P_{k}$, and 
$P_{k+1}$ (see figure \ref{mono}). Considering the triangle $T$ on $P$ given by the intersection with
the three planes, we see  that 
$$\psi_{k} + \theta_{k} \leq \psi_{k+1}.$$ 
Then, by induction, 
$$\psi_{d+1} + \sum_{i=d+1}^{m} \theta_i  \leq   \psi_{m+1}$$
As $\psi_{d+1} = \theta_{d}$, we obtain
$$|b-a| = \sum_{i=d}^m \theta_i \leq \psi_{m+1}=\psi.$$
\begin{figure}[htbp] 
   \centering
\includegraphics[width=4in]{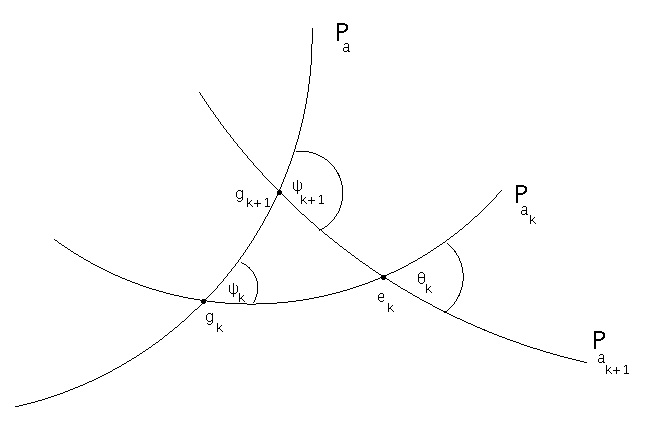}
   \caption{$ \psi_k + \theta_k \leq \psi_{k+1}$}
\label{mono}
\end{figure}

\eproof

\subsection{Facts from hyperbolic trigonometry}

The following elementary facts from hyperbolic geometry, which were
established in \cite{BC03}, will be used in the proof of Lemma \ref{thickintbound}.
 
\begin{lemma}{\rm (\cite[Lemma 3.1]{BC03})}
\label{isosceles}
Suppose that $T$ is a hyperbolic triangle with a side of length $C$ and opposite angle
equal to $\gamma$. Then the triangle 
has perimeter at most
$$C+2\sinh^{-1}\left(\frac{\sinh(C/2)}{\sin(\gamma/2)}\right)$$
and this maximal perimeter is realized uniquely by an isosceles triangle.
\end{lemma}

\begin{lemma}
\label{arcanglebound}
{\rm (\cite[Lemma 3.2]{BC03})}
Let $P_0$, $P_1$ and $P_2$ be planes in $\Ht$ such that  $\partial P_0$
and $\partial P_1$ are tangent in $\rs$
and $\partial P_1$ and $\partial P_2$ are also tangent in $\rs$.
If $L\le2\sinh^{-1}(1)$ and $\beta:[0,1]\to \Ht$ is a curve of length at most $L$
which intersects all three planes, then
$P_0$ and $P_2$ intersect with interior dihedral angle $\theta$, where 
$$\theta \geq 2 \cos^{-1}\left(\sinh(L/ 2)\right).$$
\end{lemma}

In the proof of Lemma \ref{shortgeointbound} we will need the following
extension of Lemma \ref{arcanglebound} which bounds the length
of closed curves intersecting all three sides of a triangle.

\begin{lemma}
\label{trianglefact}
Let $P_0$, $P_1$ and $P_2$ be planes in $\Ht$ such that  $\partial P_0$
and $\partial P_1$ are tangent in $\rs$
and $\partial P_1$ and $\partial P_2$ are also tangent in $\rs$.
If $P_0$ and $P_2$ intersect at an interior angle of $\theta$
and $\beta:[0,1]\to \Ht$ is a  closed curve which intersects all three planes,
then $l_h(\beta)\ge R(\theta)$, where
$$R(\theta) =  \left\{ \begin{array}{ll}
2\sinh^{-1}\left(\frac{1}{\tan(\theta/2)}\right)  &  \theta > \pi/2 \\
2\sinh^{-1}\left(\frac{\sin(\theta)}{\tan(\theta/2)}\right) &  \theta \leq \pi/2
\end{array}
\right.
$$ 
\end{lemma}

\begin{figure}[htbp] 
   \centering
\includegraphics[width=3in]{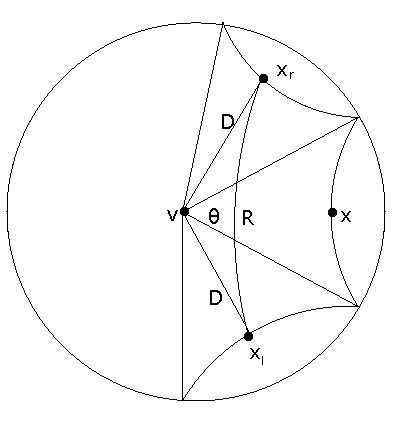}
   \caption{Triangle $T_\theta$}
\label{triangle}
\end{figure}

\begin{proof}{}
Let $P$ be the plane which intersects $P_0$, $P_1$ and $P_2$ orthogonally and let
$T$ be the triangle they bound in $P$. If $\beta'$ is the projection of $\beta$ onto $P$,
then \hbox{$l_h(\beta')\le l_h(\beta)$} and  $\beta'$ 
intersects all three sides of $T$. Therefore, it suffices
to bound the length of a curve in $P$ which intersects all three sides of $T$. Since $T$ is convex it suffices to consider curves which lie in $T$.

Let $v$ be the non-ideal vertex of $T$ and
let $D$ be length of the perpendicular $\tau$ which joins a point $x$  in
the opposite side of $T$ to $v$. 
Then, by hyperbolic trigonometry, e.g. \cite[Theorem 7.11.2]{beardon},
$$\sinh(D) = \frac{1}{\tan(\theta/2)}.$$
In $\Hp$ we take  two copies $T_l$ and $T_r$ of $T$ and glue $T_l$ to left edge of $T$ 
containing $v$ and $T_r$ to the right edge containing $v$. We label the points in the copies corresponding to $x$ by $x_l$ and $ x_r$. We can join $x_l$ and $ x_r$ by edges of length $D$ to $v$
which subtend an angle $2\theta$ at $v$ in the union of the three triangles
(see figure \ref{triangle}). If \hbox{$\theta < \pi/2$,}  we join the points $x_l$ and $x_r$
by a geodesic $g$ in the union of the three copies and let \hbox{$R = d(x_l,x_r)$.} Then we have an isoceles triangle with vertices $x_l$ , $x_r$, and $v$ and two
sides of length $D$ which make an angle $2\theta$.  We may decompose this triangle
into two right angled triangles with hypotenuse of length $D$
and again apply \cite[Theorem 7.11.2]{beardon} to  show that
$$\sinh(D) = \frac{\sinh(l_h(g)/2)}{\sin(\theta)}.$$
If $\theta\ge{\pi\over 2}$, then we join $x_l$ to
$x_r$ by a path $g$ of length $2D$ made from copies of $\tau$ in $T_l$ and $T_r$.

In either case, $g$ projects to a closed curve in $T$ meeting all three sides. To complete
the proof, we must show that $g$ projects to the minimal length such curve. If $g'$ is the minimal curve, then we may unfold it to a geodesic joining a point $x_l'$ to the side of $T_l$ containing
$x_l$ to a point $x_r'$ in the side of $T_r$ containing $x_r$. Moreover, if we join
$x_l'$ and $x_r'$ to $v$ by geodesics, the two geodesics each have length at least $D$
and make an angle of $2\theta$. It follows that the length of $g'$ is at least as great
as that of $g$, with equality if and only if $g=g'$.  
\end{proof}

\subsection{Proof of Lemma \ref{thickintbound}}
We are now prepared to prove Lemma \ref{thickintbound}, whose statement
we recall below:

\medskip
\noindent{\bf Lemma \ref{thickintbound}}{\em \ If \ $\Omega$ is the complement of a finite
collection $X$ of points in $\rs$ and $\alpha:[0,1] \rightarrow \chb$ is a  geodesic path with
$l_{h}(\alpha) \le G({\rm inj}_{\chb}(\alpha(s))),$ for some $s \in[0,1]$, then
$$i(\alpha) \le 2\pi.$$
}

\medskip

{\bf Proof:} 
Our assumptions imply that there exists $s_0\in [0,1]$ such that
${\rm inj}_{\chb}(\alpha(s_0))\le  F(l_h(\alpha))$ (since $G=F^{-1}$).
We may assume that $l_h(\alpha[0,s_0])\le l_h(\alpha)/2$.
(If this is not the case we simply consider $\bar\alpha:[0,1]\to \chb$ given
by $\bar\alpha(s)=\alpha(1-s)$ and notice that $i(\alpha)=i(\bar\alpha)$.)

We first handle the case where $\chb$ is contained in a plane. In this case,
if $i(\alpha)>2\pi$, then there exists a subarc $\alpha_1$ containing
$\alpha(s_0)$ with $i(\alpha_1)=2\pi$.
Recall that $\chb$ is the double of $CH(X)$ in this case.
Both endpoints of $\alpha_1$ must lie in the same ``copy'' of $CH(X)$.
One may then join the endpoints of $\alpha_1$ by a geodesic arc $\beta$
in this copy, so that $\delta=\alpha_1\cup\beta$ is homotopically non-trivial
in $\chb$ and $l(\delta)<2l(\alpha)$. It follows that
$${\rm inj}_{\chb}(\alpha(s_0))<l_h(\alpha)<F(l_h(\alpha))$$
which is a contradiction.

We now move on to the general case.
Let $\{ P_t\ |\ 0\le t\le i(\alpha)\}$ be the full
family of support planes to $\alpha$. There are three cases.

\medskip\noindent{\bf Case 1:} If $H_t \cap H_0 \neq \emptyset$ for all $t$,
then Lemma \ref{sptbnd} implies that  $$i(\alpha) \leq \pi.$$

If we are not in case 1, let $P_{t_1}$ be the first support plane with $H_{t_1} \cap H_0 = \emptyset$

\noindent{\bf Case 2:}  If $H_t \cap H_{t_1}$ for all $t \geq t_1$,
then by applying Lemma \ref{sptbnd} to the two families
$\{ P_t\ | \ 0\le t< t_1\}$ and $\{P_t\ |\ t_1\le t\le i(\alpha)\}$, we have
$$i(\alpha) \leq \pi + \pi = 2\pi.$$

\noindent{\bf Case 3:} If we are not in case (1) or (2), we let  $P_{t_2}$ be the first support plane such that \hbox{$H_{t_2} \cap H_{t_1} = \emptyset$.}
Consider the configuration of three planes $P_0$, $P_{t_1}$, and $P_{t_2}$.
Let \hbox{$\alpha_1=\alpha|_{[0,s_1]}$} be the subarc of $\alpha$ of minimum length having
\hbox{$\{P_t\ |\ 0 \leq t \leq t_2\}$} as a family of support planes.
Let $x$ and $y$ be the endpoints of $\alpha_1$. Lemma \ref{arcanglebound}
implies that $P_0$ and $P_{t_2}$ intersect and have dihedral angle $\theta$ satisfying
$$\theta \geq 2 \cos^{-1}(\sinh(l_h(\alpha_1)/2)).$$

We join $x$ to  $y$ by a shortest path $\beta$ on \hbox{$P_0 \cup P_{t_2}$} which intersects
\hbox{$P_0 \cap P_{t_2}$} at  a point \hbox{$z \in P_0 \cap P_{t_2}$.}
We take the triangle $T$ with vertices $x$, $y$, and $z$. 
Then the interior angle $\theta_z$ of $T$ at $z$ satisfies \hbox{$\theta_z \geq \theta$.}
Since $x$ and $y$ lie in $\alpha_1$, \hbox{$d(x,y) \leq l_h(\alpha_1)$.} 
Then, by Lemma \ref{isosceles},
$$d(x,z) + d(z,y) \leq 2\sinh^{-1}\left(\frac{\sinh(d(x,y)/2)}{\sin(\theta_z/2)}\right) \leq 2\sinh^{-1}\left(\frac{\sinh(l_h(\alpha_1)/2)}{\sin(\theta/2)}\right) .$$
Substituting the bound for $\theta$ we get
$$d(x,z) + d(z,y) \leq  2\sinh^{-1}\left(\frac{\sinh(l_h(\alpha_1)/2)}{\sqrt{1-\sinh^2(l_h(\alpha_1)/2)}}\right) .$$
If $\delta=\alpha_1\cup\beta$, then
$$l_h(\delta) \leq l_h(\alpha_1) + 2\sinh^{-1}\left(\frac{\sinh\left( {l_h(\alpha_1)\over 2}\right)}{\sqrt{1-\sinh^{2}\left({l_h(\alpha_1)\over 2}\right)}}\right) = 2F(l_h(\alpha_1)).$$

We next show that $r(\delta)$ is homotopically non-trivial on $\chb$. Since the restriction
of $r$ to \hbox{$\Ht-{\rm int}(CH(X)))$} is a homotopy equivalence onto $\chb$, it
suffices to show that $\delta$ is homotopically non-trivial in \hbox{$\Ht-{\rm int}(CH(X)).$}
Let $b=\alpha(s_1)$ be the first point on $\alpha$ which has support plane $P_{t_1}$.
Then $b$ lies on an edge $e$ of $\chb$ which is contained in $P_{t_1}$. Let $Q_b$
be the plane through $b$ which is perpendicular to $P_{t_1}$ and let $R_b$ be
the intersection of $Q_b$ with the closed half-space bounded by $P_{t_1}$ which
is disjoint from ${\rm int}(CH(X))$. Lemma \ref{sptbnd} implies that $\alpha_1$
intersects $b$ exactly once at $\alpha(s_1)$ and the intersection is transverse.
By construction, $\beta$ cannot intersect $R_b$ and $\alpha_1$ cannot intersect
$R_b-b$. Therefore, $\delta$ intersects $R_b$ exactly once and does so transversely.
Since $R_b$ is a properly embedded half-plane in \hbox{$\Ht-{\rm int}(CH(X))$,} we see that
 $\delta$ is homotopically non-trivial in \hbox{$\Ht-{\rm int}(CH(X))$} as desired.

Since $r$ is 1-Lipschitz on $\Ht$, 
$$l_h(r(\delta))\le l_h(\delta)\le 2F(l_h(\alpha_1))<2F(l_h(\alpha)).$$
In particular, since $r(\delta)$ is homotopically non-trivial and passes
through $\alpha(0)$, we see that
$${\rm inj}_{\chb}(x)\le F(l_h(\alpha))$$
for all $x\in\alpha_1$. If \hbox{$\alpha(s_1)\in \alpha_1$}, i.e. if \hbox{$s_0\le s_1$,} then we
have achieved a contradiction and we are done.
If $\alpha(s_0)$ does not lie in $\alpha_1$, then let \hbox{$\alpha_2=\alpha|_{[s_1,s_0]}$.}
Then \hbox{$\gamma=\bar\alpha_2 * \bar r(\beta) * \alpha_1 *\alpha_2$} is a homotopically
non-trivial loop through $\alpha(s_0)$ such that
$$l_h(\gamma)=2l_h(\alpha_2)+l_h(r(\delta)) < l_h(\alpha)+2F(l_h(\alpha_1)).$$
Since
$$l_h(\alpha)+2F(l_h(\alpha_1))<l_h(\alpha)+2F(l_h(\alpha)/2)<
2F(l_h(\alpha))$$
we have again achieved a contradiction.
\eproof

\subsection{Proof of Lemma \ref{shortgeointbound}}

We now establish Lemma \ref{shortgeointbound} which bounds intersection number
for short simple closed geodesics.

\medskip\noindent {\bf Lemma \ref{shortgeointbound}}
{\em Suppose that $\Omega$ is the complement of finite collection $X$ of points in $\rs$. 
If $\gamma$  is a simple closed geodesic of length less than $2\sinh^{-1}(1)$, then
$$2\pi\le i(\gamma)  \leq 2\pi + 2\tan^{-1}(\sinh(l_h(\gamma)/2)) \leq \frac{5\pi}{2} $$
}
\begin{proof}
We first note that if  $\chb$ lies in a plane and $\gamma$ is a simple closed geodesic
on $\chb$, then 
$\gamma$ is the double of a simple geodesic arc in $CH(X)$, so $i(\gamma) = 2\pi$.  

For the remainder of the proof, we assume that $CH(X)$ does not lie in a plane.
We choose a parameterization \hbox{$\alpha: [0,1] \rightarrow \chb$} of $\gamma$ which is an embedding on $(0,1)$ and such that \hbox{$\alpha(0)=\alpha(1)$} does not lie on an edge of $\chb$.
Recall that Lemma \ref{intlowerbound} implies that \hbox{$i(\gamma)\ge 2\pi$.}

Let \hbox{$\{ P_t\ |\ 0\le t\le i(\alpha)\}$} be the full
family of support planes to $\alpha$.
If we proceed
as in the proof of Lemma \ref{thickintbound}, we must be in case (3).
Therefore, we obtain support planes $P_0$, $P_{t_1}$ and $P_{t_2}$,
such that $\partial P_0$ and $\partial P_{t_1}$ are tangent on $\rs$, as
are $P_{t_1}$ and $P_{t_2}$, and a subarc
$\alpha_1$ of $\alpha$ which begins at \hbox{$\alpha(0)\in P_0$,} ends at
\hbox{$\alpha(s_1)\in P_{t_2}$} and intersects $P_{t_1}$.  Lemma \ref{arcanglebound} implies
that $P_0$ and $P_{t_2}$ intersect. Lemma \ref{trianglefact} implies that if
$\theta$ is the angle of intersection of $P_0$ and $P_{t_2}$, then
\hbox{$\theta \geq R^{-1}(l_h(\gamma)).$}
As  $R$ is decreasing and \hbox{$R(\pi/2) = 2\sinh^{-1}(1) \geq l_h(\gamma)$}
we see that  \hbox{$\theta \geq \pi/2$.}

Therefore,
$$\theta \geq R^{-1}(l_h(\gamma)) = 2\tan^{-1}\left(\frac{1}{\sinh\left(\frac{l_h(\gamma)}{2}\right)}\right)$$
Let  $\phi = \pi-\theta \leq \pi/2$ be the exterior dihedral angle between $P_0$ and
$P_{t_2}$. Then,
$$\phi \leq 2\tan^{-1}(\sinh(l_h(\gamma)/2)).$$

Let $\beta_0$ be
the shortest path  on \hbox{$P_0 \cup P_{t_2}$}  joining $\alpha(0)$ to $\alpha(s_1)$
and let $z_0$ be the point of intersection of $\beta_0$ with \hbox{$P_0 \cap P_{t_2}$.} 
Since
$$d_h(\alpha(0),\alpha(s_1))\le l_h(\gamma)/2,$$ 
Lemma \ref{isosceles} implies that
$$l_h(\beta_0) \leq 2 \sinh^{-1}\left(\frac{\sinh(d_h(\alpha(0),\alpha(s_1))/2)}{\sin(\theta/2)}\right),$$
so
$$\sinh(l_h(\beta_0)/2) \leq \frac{\sinh(l_h(\gamma)/4)}{\sin(\theta/2)}.$$
Since $\theta \ge {\pi\over 2}$ and $l_h(\gamma)<2\sinh^{-1}(1)$,
we see that
$$\sinh(l_h(\beta_0)/2) \leq \frac{\sinh(\frac{\sinh^{-1}(1)}{2})}{\sin(\frac{\pi}{4})}\approx .6436  <1.$$
It follows that $l_h(\beta_0)/2<\sinh^{-1}(1)$.

We let \hbox{$\delta_0=\alpha_1\cup\beta_0$.} We may argue, exactly as in the proof of
Lemma \ref{shortgeointbound}, that $r(\delta_0)$ is homotopically non-trivial on $\chb$.
Since \hbox{$\alpha_1\subset\gamma$,} \hbox{$l_h(\beta_0)/2<\sinh^{-1}(1)$} and
$N(\gamma)$ has width \hbox{$w(\gamma)>\sinh^{-1}(1)$,} we see that $r(\delta_0)$ is
contained in $N(\gamma)$, so is
homotopic to a non-trivial power of $\gamma$.

We now describe a process to replace $\beta_0$ by a path $\beta_n$ on
$\chb$ such that \hbox{$i(\beta_n) \leq \phi $} and $\beta_n$ is homotopic to $r(\beta_0)$.
Once we have done so, we can form
\hbox{$\delta_n = \alpha_1\cup \beta_n$} and observe that
$$i(\delta_n) \leq i(\alpha_1) + \phi \le 2\pi+\phi\le
2\pi+ 2 \tan^{-1}(\sinh(l_h(\alpha_1)/2))\le{5\pi\over 2}.$$ 
Our result follows,
since $\delta_n$ is homotopic to a non-trivial power of $\gamma$ and a geodesic
minimizes intersection number in its homotopy class.

We label $Q_0 = P_{t_2}$ and let \hbox{$P_{0}\cap Q_{0} = h_0$.} If $h_0$ is an edge then $P_0$ and $Q_{0}$ are adjacent and $\beta_0$ is on $\chb$ and we are done. 

Otherwise $\beta_0$ intersects an edge on $Q_{0}$ labelled $e_0$, such that $h_0$ and
$e_0$ are disjoint. We choose $e_0$ to be the last edge on $Q_0$ that $\alpha$ intersects.
Let $y_0$ be the first point of intersection of  $\beta_0$ and $e_0$.  Let $F_1$ be the face
of $\chb$ which abuts $e_0$ and does not contain a subarc of $\beta_0$. Let $Q_1$ be the support plane containing $F_1$.

Define  $h_1 = P_0 \cap Q_1$.  Since $h_0$ and $h_1$ are both contained in
support planes to $e_0$, they are disjoint.  Since $\alpha(0)\in \chb$, $h_0$  must separate $h_1$ from $\alpha(0)$ in $P_0$.
Therefore, $\beta_0$ intersects $h_1$ at a point $z_1$ between $z_0$ and $\alpha(0)$.
We construct $\beta_1$ by replacing the subarc of $\beta_0$ joining $y_0$ to $z_1$ by  the geodesic  arc in $F_1$ joining $y_0$ to $z_1$. Notice that $r(\beta_1)$ is
homotopic to $r(\beta_0)$ in $\chb$.
If we let $\theta_1$ be the exterior dihedral angle between $Q_0$ and $Q_1$ and let
$\phi_1$ be the exterior dihedral angle between $Q_1$ and $P_0$, then
$$\theta_1 + \phi_1 \leq \phi.$$

If $h_1$ is an edge, then
$\beta_1$ lies on $\chb$ and we are done since 
$$i(\beta_1) = \theta_1  + \phi_1 \leq \phi.$$

\begin{figure}[htbp] 
   \centering
\includegraphics[width=4in]{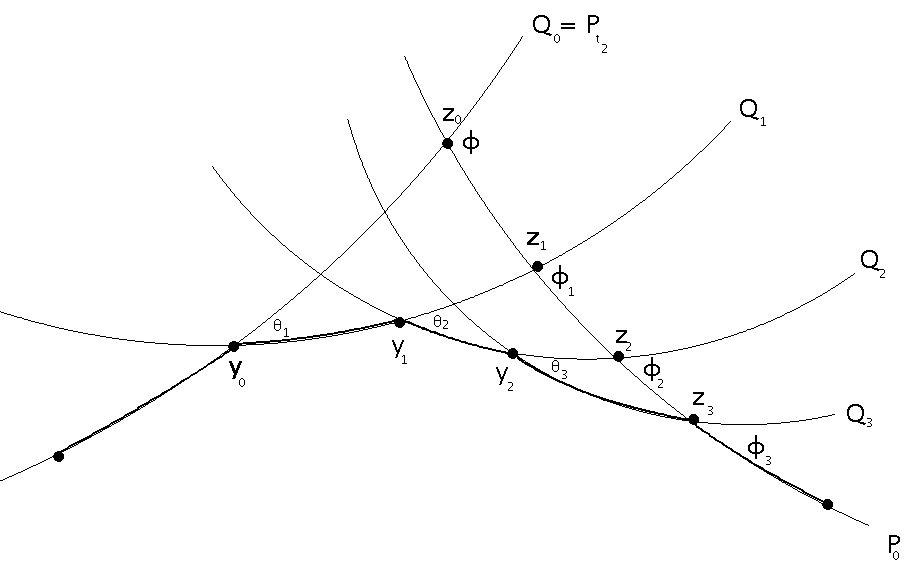}
   \caption{$\beta_n = \beta_3$}
\label{project}
\end{figure}

If $h_1$ is not an edge, we define $\beta_n$
by an iterative procedure.
Assume $Q_k$ are defined as above, for \hbox{$k = 0,\ldots, m$,} such that 
$h_{m} = Q_{m}\cap P_0$ is not an edge. Define $\theta_k $ to be the exterior dihedral angle between $Q_{k-1}$ and $Q_k$ and $\phi_k$  to be the exterior dihedral angle between
$Q_k$ and $P_0$. We assume, by induction, that
$$\sum_{k=1}^m \theta_k + \phi_m \leq \phi.$$
Let $h_m$ be the last edge on $Q_m$ which $\beta_m$ intersects, and  let $y_m$
be the first point of intersection of $\beta_m$ and $h_m$.
Let $F_{m+1}$ be the face adjacent to $e_m$ which does not agree with $F_m$
and let $Q_{m+1}$ be the support plane it is contained in.
As before, if we let \hbox{$h_{m+1}=Q_{m+1}\cap P_0$,} then $h_{m+1}$ separates $h_m$
from $\alpha(0)$.
Therefore,  $\beta_m$ intersects $h_{m+1}$ in a point $z_{m+1}$.
We construct $\beta_{m+1}$ by replacing the piecewise geodesic subarc of $\beta_m$ connecting $y_{m+1}$ to $z_{m+1}$ by the geodesic  arc in $F_{m+1}$ joining $y_{m+1}$
to $z_{m+1}$.  Again, notice that $r(\beta_{m+1})$ is homotopic to $r(\beta_m)$.
Furthermore,
$$\theta_{m+1} +\phi_{m+1} \leq \phi_m.$$ 
Therefore, 
$$\sum_{k=1}^{m+1}\theta_k + \phi_{m+1} = \sum_{k=1}^n \theta_k + (\theta_{m+1} + \phi_{m+1}) \leq \sum_{k=1}^m \theta_m + \phi_{m}  \leq \phi.$$

Notice that the support planes $\{Q_k\}$ in this procedure have disjoint
intersection with $P_0$ and each contain a face $F_k$. Therefore, as there are only a finite number of faces, the process terminates and some $h_n$ is an edge of $\chb$.
Then, $\beta_n$ lies on $\chb$ and our argument is complete.
\end{proof}

\end{document}